\documentclass[reqno]{amsart}

\usepackage{amsfonts}
\usepackage{bbm}
\usepackage{mathrsfs}
\usepackage{mathrsfs}
\usepackage{mathrsfs}
\usepackage{bbm}
\usepackage{mathrsfs}
\usepackage{mathrsfs}
\usepackage{bbm}
\usepackage{amsfonts,amssymb,amsmath,amsthm}
\usepackage{mathrsfs}
\usepackage{graphicx}
\usepackage{url}
\usepackage{enumerate}

\newtheorem{theorem}{Theorem}[section]
\newtheorem{lemma}[theorem]{Lemma}
\newtheorem{proposition}[theorem]{Proposition}
\newtheorem{corollary}[theorem]{Corollary}
\theoremstyle{definition}
\newtheorem{definition}[theorem]{Definition}
\newtheorem{remark}{Remark}
\newtheorem{example}{Example}

\numberwithin{equation}{section}

\DeclareMathOperator{\diam}{diam}

\DeclareMathOperator{\vol}{Vol}

\author{Wei Zhao}
\address{
Department of Mathematics\\
East China University of Science and Technology\\
Shanghai, China}
\email{szhao\underline{ }wei@yahoo.com}
\author{Yibing Shen}
\address{
 Center of Mathematical Sciences\\
Zhejiang University\\
Hangzhou, China}
\email{yibingshen@zju.edu.cn}

\thanks{This work was supported by National
Natural Science Foundation of China, Tian Yuan Foundation (No. 11426108) and
the Fundamental Research Funds for the Central Universities}

\begin{document}

\title[]{A Cheeger type finiteness theorem for Finsler manifolds}

\begin{abstract}
In this paper, we mainly establish a Cheeger type finiteness theorem for Berwald manifolds.  In order to do this, we study the injectivity radius and the convex radius of a Finsler manifold. A Cheeger type estimate on injectivity radii for Finsler manifolds is given and the existence of the center of mass of a Berwald manifold is proved.
\end{abstract}
\maketitle

\section{Introduction}
The estimate of injectivity radius plays an important role in global differential geometry. For a compact Riemannian manifold, Klingenberg \cite{K} gives a lower bound for the injectivity radius in terms of an upper bound for the sectional curvature and a lower bound for the
length of simple closed geodesics. And Cheeger in \cite{Ch} points out there exists a lower bound for the
length of simple closed geodesics, which together with Klingenberg's result yields a positive constant $c_n(k,D,V)$ such that if an arbitrary compact Riemannain $n$-manfiold satisfies $|\mathbf{K}_M|\leq k$, $\diam(M)\leq D$ and $\vol(M)\geq V$, then the injectivity radius $\mathfrak{i}_M\geq c_n(k,D,V)$.
Eight years later, Heintze and Karcher in \cite{HK} give the explicit expression of $c_n(k,D,V)$ by their volume comparison theorem.  Refer to \cite{AM,C2,P} for more details.

Finsler geometry is just Riemannian geometry without quadratic restriction.
It is an natural question that whether an analogue of the above estimate
still holds in the Finslerian case.
To answer this question, we introduce some non-Riemannian geometric quantities first: Given a Finsler manifold, let $\Lambda_F$ and $\mathbf{T}_M$ denotes its  uniformity constant and T-curvature, respectively (see \cite{E,Sh1} or Sect.\,2,3 for the definitions). It should be remarked that $\Lambda_F=1$ if and only if $F$ is Riemannian, while $\mathbf{T}_M=0$ if and only if $F$ is Berwalden.
We then shall establish the following estimate.
\begin{theorem}\label{injectiveestimeat}
Let $(M,F)$ be a compact Finsler $n$-manifold with $|\mathbf{K}_M|\leq k$, $\mathbf{T}_M\leq \tau$, $\Lambda_F\leq\Lambda$, $\diam(M)\leq D$ and $\mu(M)\geq V$, where $\mu(M)$ is either the Busemann-Hausdorff volume or the Holmes-Thompson volume of $M$. Then
\[
\mathfrak{i}_M\geq\frac{1}{1+\Lambda^{\frac12}}\min\left\{\frac{(1+\Lambda^{-\frac12})\pi}{\sqrt{k}},\,\frac{\mu(M)}{c_{n-2}\Lambda^\frac{3n}{2}\left[\frac{\mathfrak{s}^{n-1}_{-k}\left(D\right)}{n-1}+\Lambda^{\frac12}\tau\int_0^{D}\mathfrak{s}^{n-1}_{-k}(t) dt\right]}\right\}.
\]
\end{theorem}
The theorem above not only implies the estimate in the Riemannian case, but also points out that the injectivity radius is inversely proportional to  the uniformity constant. In fact, we have the following non-Riemannian example. Also refer to \cite{Z2} for more examples.
\begin{example}Define a sequence of Berwald metrics on $\mathbb{T}^2=\mathbb{S}^1\times \mathbb{S}^1$ by
\[
F_n:=\alpha+\left(1-\frac{1}{n}\right)\beta,\ n\geq 1,
\]
where $\alpha$ is the canonical Riemannian product metric
on $\mathbb{T}^2$, and $\beta$ is a parallel 1-form on $\mathbb{T}^2$ with $\|\beta\|_\alpha=1$. Then $\{(\mathbb{T}^2,F_n)\}_n$ satisfy
\[
\mathbf{K}_n=0,\  \diam_n(\mathbb{T}^2)\leq 2(\sqrt{2}+1)\pi, \ \mu_n(\mathbb{T}^2)=4\pi^2,
\]
where $\mu_n$ is the  Holmes-Thompson volume of $(\mathbb{T}^2,F_n)$.
However, the injectivity radius $\mathfrak{i}_n(\mathbb{T}^2)\rightarrow 0$ while the uniformity constant $\Lambda_n\rightarrow +\infty$  as $n\rightarrow +\infty$.
\end{example}

From above, one can see that a Berwald manifold cannot collapse if we control  the lower bound of the volume,  the upper bounds of the diameter, the uniformity constant and the bounds of the flag curvature. Thus, according to \cite{Ch,Pe}, it is an natural question that whether the class of such Berwald $n$-manifolds is finite up to  homeomorphism or diffeomorphism?
The answer is affirmative. In fact, we shall establish
the following Cheeger type finiteness theorem. Refer to \cite{Sh1,YZ,Z} for other finiteness theorems for Finsler manifolds.
\begin{theorem}\label{Cheegetype}
Given $n\in\mathbb{N}$, $\Lambda\geq 1$, $k\geq 0$ and $V,D>0$, there exist only finitely many diffeomorphism classes of compact Berwald
$n$-manifolds $(M,F)$ satisfying
\[
\Lambda_F\leq \Lambda,\ |\mathbf{K}_M|\leq k,\  \mu(M)\geq V,\ \diam(M)\leq D,
\]
where $\mu(M)$ is either the Busemann-Hausdorff volume or the Holmes-Thompson volume of $M$.
\end{theorem}

The arrangement of contents of this paper is as follows. In Sect.\,2, we brief some
necessary definitions and properties concerned with Finsler geometry. In Sect.\,3, a Finslerian version of Klingenberg's theorem is established and Theorem \ref{injectiveestimeat} is proved. In Sect.\,4, we estimate the convex radius and study the center of mass of a Berwald manifold. Theorem 1.2 is proved in Sect.\,5 by a generalized Peter's lemma, and the latter is proved in Sect.\,6. In App.\,A, we give some estimates for Jacobi fields on Finsler manifolds. In App.\,B, we study the parallel transformations on a Berwald manifold.

\section{Preliminaries}
In this section, we recall some definitions and properties about Finsler manifolds. See \cite{BCS,Sh1} for more details.

Let $(M,F)$ be a (connected) Finsler $m$-manifold with Finsler metric
$F:TM\rightarrow [0,\infty)$. Define $S_xM:=\{y\in T_xM:F(x,y)=1\}$
and $SM:={\cup}_{x\in M}S_xM$. Let $(x,y)=(x^i,y^i)$ be
local coordinates on $TM$. Define
\begin{align*}
&\ell^i:=\frac{y^i}{F},\ g_{ij}(x,y):=\frac12\frac{\partial^2
F^2(x,y)}{\partial y^i\partial
y^j},&A_{ijk}(x,y):=\frac{F}{4}\frac{\partial^3 F^2(x,y)}{\partial
y^i\partial y^j\partial y^k},\\
&\gamma^i_{jk}:=\frac12 g^{il}\left(\frac{\partial g_{jl}}{\partial
x^k}+\frac{\partial g_{kl}}{\partial x^j}-\frac{\partial
g_{jk}}{\partial x^l}\right),
&N^i_j:=\left(\gamma^i_{jk}\ell^j-A^i_{jk}\gamma^k_{rs}\ell^r
\ell^s\right)\cdot F.
\end{align*}
The Chern connection $\nabla$ is defined on the pulled-back bundle
$\pi^*TM$ and its forms are characterized by the following structure
equations:

(1) Torsion freeness: $dx^j\wedge\omega^i_j=0$;

(2) Almost $g$-compatibility: $d
g_{ij}-g_{kj}\omega^k_i-g_{ik}\omega^k_j=2\frac{A_{ijk}}{F}(dy^k+N^k_l
dx^l)$.

From above, it's easy to obtain $\omega^i_j=\Gamma^i_{jk}dx^k$, and
$\Gamma^i_{jk}=\Gamma^i_{kj}$. It should be remarked that $\Gamma^i_{kj}=\Gamma^i_{kj}(x,y)$ is a local smooth function on $SM$. In particular, $F$ is called a Berwald metric if $\frac{\partial\Gamma^i_{kj}}{\partial y^s}=0$.

The curvature form of the Chern connection is defined as
\[
\Omega^i_j:=d\omega^i_j-\omega^k_j\wedge\omega^i_k=:\frac{1}{2}R^i_{j\,kl}d
x^k\wedge d x^l+P^i_{j\,kl}d x^k\wedge\frac{d y^l+N^l_s dx^s}{F}.
\]
Given a non-zero vector $V\in T_xM$, the flag curvature $K(y,V)$ on
$(x,y)\in TM\backslash 0$ is defined as
\[
\mathbf{K}(y,V):=\frac{V^i y^jR_{jikl}y^l
V^k}{g_y(y,y)g_y(V,V)-[g_y(y,V)]^2},
\]
where $R_{jikl}:=g_{is}R^s_{j\,kl}$.

The reversibility $\lambda_F$ and the uniformity constant $\Lambda_F$ of
$(M,F)$ are defined as
\[
\lambda_F:=\underset{X\in TM\backslash0}{\sup}\frac{F(-X)}{F(X)},\ \Lambda_F:=\underset{X,Y,Z\in SM}{\sup}\frac{g_X(Y,Y)}{g_Z(Y,Y)}.
\]
Clearly, $\lambda_F\geq 1$ with equality if and only if $F$
is reversible, and $\Lambda_F\geq 1$ with equality if and only if $F$ is Riemannian. In particular, $\lambda_F\leq \sqrt{\Lambda_F}$.

The Legendre transformation $\mathcal {L} : TM
\rightarrow T^*M$ is defined by
\[\mathcal {L}(Y)=
\left \{
\begin{array}{lll}
&0, &Y=0,\\
&g_Y(Y,\cdot),&Y\neq 0.
\end{array}
\right.
\]
For each $x\in M$, the Legendre transformation is a smooth
diffeomorphism from $T_xM\backslash\{0\}$ to
$T^*_xM\backslash\{0\}$.

The average Riemannian metric $\tilde{g}$ induced by $F$ is defined by
\[
\tilde{g}(X,Y)=\frac{1}{\vol(x)}\int_{S_xM}g_y(X,Y)d\nu_x(y),\ \forall\,X,Y\in T_xM,
\]
where $\vol(x):=\int_{S_xM}d\nu_x(y)$.

\section{The injectivity radius of a compact Finsler manifold}
For a general Finsler metric, the Legendre transformation is non-linear but only positive homogeneous. First, we show the following result, which is claimed in \cite{R} without proof.
\begin{lemma}\label{threevectors}
Let $(M,F)$ be a Finsler $n$-manifold. Given three distinct vectors $X,Y,Z\in S_pM$, we have
\[
\dim\{W\in T_pM: \ \mathcal {L}_X(W)=\mathcal {L}_Y(W)=\mathcal {L}_Z(W)\}\leq n-2.
\]
\end{lemma}
\begin{proof}Set $A:=\mathcal {L}_X$, $B:=\mathcal {L}_Y$ and $C:=\mathcal {L}_Z$. After choosing a basis for $T_pM$, we can view $A,B,C$ as three vectors in $(\mathbb{R}^n,\langle\cdot,\cdot\rangle)$, i.e., $\mathcal {L}_X(W)=\langle A,W\rangle$,
where $\langle\cdot,\cdot\rangle$ is a standard Euclidean inner product.
 Consider the solution space $\mathcal {S}$ of following system:
\[
\left \{
\begin{array}{lll}
\langle(A-B),W\rangle=0,\\
\langle(A-C),W\rangle=0.
\end{array}
\right.
\]
Clearly,  $\mathcal {S}=(A-B)^\perp\cap  (A-C)^\perp$, where $(\cdot)^\perp$ denote the orthogonal complement of $(\cdot)$ in $(\mathbb{R}^n,\langle\cdot,\cdot\rangle)$. Since $\mathcal {L}$ is injective, $\dim(\mathcal {S})\leq n-1$.

 Suppose that $\dim(\mathcal {S})= n-1$. Thus, $\mathcal {S}=(A-B)^\perp=(A-C)^\perp$ and hence, there exists a nonzero constant $\alpha$ such that $A-B=\alpha(A-C)$. Clearly, $\alpha\neq1$.

\noindent\textbf{Case 1.} Suppose $\alpha>0$. Since $(\alpha-1)A=\alpha C-B$, we can assume that $\alpha >1$. Thus,
\[
\alpha-1=F((\alpha-1)A)=F(\alpha C-B)\geq \alpha F(C)-F(B)=\alpha-1.
\]
Then the triangle inequality \cite[Theorem 1.2.2]{BCS} yields $\alpha C-B=\beta B$, where $\beta\geq 0$. Hence, $\alpha=1+\beta$ and $C=B$, which is a contradiction.

\noindent\textbf{Case 2.} Suppose $\alpha<0$. Thus, $(|\alpha|+1)A=|\alpha| C+B$ and
\[
|\alpha|+1=F((|\alpha|+1)A)=F(|\alpha| C+B)\leq |\alpha| F(C)+F(B)=|\alpha|+1.
\]
The same argument as above yields $C=B$, which is a contradiction as well.

Therefore, $\dim(\mathcal {S})\leq n-2$. We are done by $\mathcal {S}=\{W\in T_pM: \ \mathcal {L}_X(W)=\mathcal {L}_Y(W)=\mathcal {L}_Z(W)\}$.
\end{proof}

Let $(M,F)$ be a compact Finsler manifold. There exist two point $p$ and $q$ such that
\[
\mathfrak{i}_M=d(p,q)=d(p,\text{Cut}_p).
\]
Let $\gamma_y(t)$, $0\leq t\leq d(p,q)$ be a normal minimal geodesic from $p$ to $q$. Then $q$ is the cut point of $p$ along $\gamma_y$.
If $q$ is not the first conjugate point of $p$ along $\gamma_y$, \cite[Proposition 8.2.1]{BCS} implies that there exists another distinct normal minimal  geodesic $\gamma_w(t)$,  $0\leq t\leq d(p,q)$ from $p$ to $q$. If $F$ is reversible, Shen \cite[Lemma 12.2.5]{Sh1} shows that $\dot{\gamma}_y(d(p,q))=-\dot{\gamma}_w(d(p,q))$. If $F$ is nonreversible,
Rademacher in \cite{R} obtains the following result.
\begin{lemma}[\cite{R}]\label{hypersur}
There exists a local hypersurface $H$ with $q\in H$ such that for each smooth curve $\sigma:(-\epsilon,\epsilon)\rightarrow H$ with $\sigma(0)=q$, there are two geodesic
variations $c_{y,s}$, $c_{w,s}:[0,d(p,q)]\rightarrow M$ with $c_{y,s}(d(p,q))=c_{w,s}(d(p,q))=\sigma(s)$, $L(c_{y,s})=L(c_{w,s})$, and $c_{y,0}(t)=\gamma_y(t)$ and $c_{w,0}(t)=\gamma_w(t)$.
\end{lemma}
\begin{remark}
The first variation formula \cite[p.123]{BCS} together with Lemma \ref{hypersur} yields that $\mathcal {L}_{\dot{\gamma}_y(d(p,q))}=-\mathcal {L}_{\dot{\gamma}_w(d(p,q))}$ is a unit normal vector of $T_qN$. However, if $F$ is nonreversible, one cannot deduce $\dot{\gamma}_y(d(p,q))=-\dot{\gamma}_w(d(p,q))$, since the Legendre transformation is non-linear.
\end{remark}

In the following, we use a method of Rademacher \cite{R,R2} to show the following theorem, which is a Finslerian version of Klingenberg's theorem.
\begin{theorem}\label{closedgeodesci}
Let $(M,F)$ be a compact Finsler manifold with $\mathbf{K}_M\leq k$. Then
\[
{\mathfrak{i}}_M\geq\min\left\{\frac{\pi}{\lambda_F\sqrt{k}},\,\frac{1}{1+\lambda_F}\text{the shortest simple closed geodesic in }M \right\}.
\]
In particular, the equality holds if $F$ is reversible (i.e., $\lambda_F=1$).
\end{theorem}
\begin{proof}
\noindent\textbf{Step 1.} As in \cite{R3,R},
 set
\[
\tilde{i}_p:=\inf\{\tilde{d}(p,q):\,q\in \text{Cut}_p\},\ \tilde{\mathfrak{i}}_M:=\inf_{p\in M}\tilde{i}_p,
\]
where $\tilde{d}(p,q):=\frac12(d(p,q)+d(q,p))$. Clearly,
\[
\frac{1+\lambda_F^{-1}}{2}\mathfrak{i}_M\leq \tilde{\mathfrak{i}}_M\leq \frac{(1+\lambda_F)}2\mathfrak{i}_M.\tag{3.1}\label{new**0}
\]
Given a closed geodesic $c(t)$, $0\leq t\leq 1$, let $c(t_0)$ denote the cut point of $c(0)$ along $c$. Then we have
\[
L(c)=d(c(0),c(t_0))+L(c|_{[t_0,1]})\geq d(c(0),c(t_0))+d(c(t_0),c(0))\geq 2\,\tilde{\mathfrak{i}}_M.\tag{3.2}\label{new3.2}
\]
Hence,
(\ref{new**0}) together with (\ref{new3.2}) implies that in order to prove the theorem,  we just need to show that there exists a simple closed geodesic $c$ with $L(c)=2\,\tilde{\mathfrak{i}}_M$ in the case of $\tilde{\mathfrak{i}}_M<\frac{(1+\lambda^{-1}_F)\pi}{2\sqrt{k}}$.

\noindent\textbf{Step 2.} Suppose $\tilde{\mathfrak{i}}_M<\frac{(1+\lambda^{-1}_F)\pi}{2\sqrt{k}}$. We now construct a simple geodesic loop $c$ (based at $c(0)$) with $L(c)=2\,\tilde{\mathfrak{i}}_M$.
Since $M$ is compact, there is a point $p\in M$ with $\tilde{i}_p=\tilde{\mathfrak{i}}_M$.
Let $q$ be the point in $\text{Cut}_p$ with $\tilde{d}(p,q)=\tilde{i}_p$.
Since
\[
d(p,q)\leq \frac{2}{(1+\lambda_F^{-1})}\tilde{d}(p,q)=\frac{2}{(1+\lambda_F^{-1})}\tilde{i}_p<\frac{\pi}{\sqrt{k}},
\]
$q$ is not the conjugate point of $p$. Thus, there exist two distinct normal minimal geodesics $c_1(t)$ and $c_2(t)$, $t\in [0,d(p,q)]$ from $p$ to $q$.
Let $c_3(t)$, $t\in[0,d(q,p)]$ be a normal minimal geodesic from $q$ to $p$.

The proof of \cite[Lemma 9.4]{R} implies that
$c_1*c_3$ or $c_2*c_3$ is smooth at $q$. For completeness, we give a sketch of this proof. Suppose that neither $c_1*c_3$ nor $c_2*c_3$ is smooth at $q$.
That is, $c'_1(d(p,q))\neq c'_3(0)$ and $c'_2(d(p,q))\neq c'_3(0)$. It follows from Lemma \ref{threevectors} that
\[
\dim\{W\in T_qM:\,\mathcal {L}_{c'_1(d(p,q))}(W)=\mathcal {L}_{c'_2(d(p,q))}(W)=\mathcal {L}_{c'_3(0)}(W)    \}\leq n-2.
\]
Denote by $H$ the hypersurface though $q$ as in Lemma \ref{hypersur}.
Since $\dim H=n-1$, Lemma \ref{threevectors} implies that there is a $v\in T_qN$ such that for $i=1$ or $i=2$,
\[
\mathcal {L}_{c'_i(d(p,q))}(v)\neq\mathcal {L}_{c'_3(0)}(v).
\]
Without loss of generality, we assume that $i=1$. Set
\[
c(t)=c_1*c_3(t),\ t\in [0,d(p,q)+d(q,p)].
\]
Let $\iota(s)$, $s\in (-\epsilon,\epsilon)$ be a smooth curve in $H$ with $\iota(0)=q$ and $\iota'(0)=v$. Let $c_{i,s}(t)$, $s\in (-\epsilon,\epsilon)$, $t\in [0,d(p,q)]$ be the geodesic variations defined as in Lemma \ref{hypersur}. Let $c_{3,s}(t)$, $s\in (-\epsilon,\epsilon)$, $t\in [0,d(q,p)]$ be the minimal geodesic variation from $\iota(s)$ to $p$. Consider the variation $c_s(t):=c_{1,s}*c_{3,s}(t)$, $s\in (-\epsilon,\epsilon)$, $t\in [0,d(p,q)+d(q,p)]$. The first variation formula then yields
\[
\left.\frac{d}{ds}\right|_{s=0}L(c_s)=g_{c'_3(0)}(c'_3(0),v)-g_{c'_1(0)}(c'_1(0),v)\neq 0.
\]
Without loss of generality, we can assume that $\left.\frac{d}{ds}\right|_{s=0}L(c_s)<0$ (otherwise, consider $\iota(-s)$). Then there exists $s_0>0$ such that $L(c_{s_0})<L(c_0)=L(c)$.
None of $c_{i,s_0}(t)$, $i=1,2$ is minimal on $[0,d(p,q)+\varepsilon]$ for any $\varepsilon>0$, since $c_{i,s_0}(d(p,q))=\iota(s_0)$. Then there is $t_0\in (0,d(p,q)]$ such that $q_1=c_{1,s_0}(t_0)$ is the cut point of $p$ along $c_{1,s_0}(t)$. Clearly, $d(p,q_1)=L(c_{1,s}|_{[0,t_0]})$.
Thus,
\begin{align*}
2\tilde{d}(p,q_1)=&d(p,q_1)+d(q_1,p)\leq d(p,q_1)+d(q_1, \iota(s_0))+d(\iota(s_0),p)\\
=&L(c_{1,s_0}|_{[0,t_0]})+d(q_1, \iota(s_0))+d(\iota(s_0),p)\\
\leq &L(c_{1,s_0}|_{[0,t_0]})+L(c_{1,s_0}|_{[t_0,d(p,q)]})+L(c_{3,s_0})\\
=&L(c_{1,s_0})+L(c_{3,s_0})=L(c_{s_0})<L(c)=2\tilde{d}(p,q),
\end{align*}
which contradicts the definition of $q$. Hence, $c_1*c_3$ is smooth at $q$ and therefore, $c$ is a geodesic loop based at $p$ though $q$ with $L(c)= d(p,q)+d(q,p)=2\tilde{d}(p,q)=2\,\tilde{\mathfrak{i}}_M$.

\noindent \textbf{Step 3.} We now show that $c=c_1*c_3$ is a closed geodesic. From above, one see that $c(t)$, $t\in {[d(p,q),2\,\tilde{\mathfrak{i}}_M]}$ is a minimal geodesic from $q$ to $p$.

The continuity of the cut value \cite[Proposition 8.4.1]{BCS} implies that
for a small positive number $\varepsilon(<d(p,q))$, there exists $t_\varepsilon\in (d(p,q),2\,\tilde{\mathfrak{i}}_M)$ such that $q_\varepsilon=c(t_\varepsilon)$ is a cut point of $p_\varepsilon=c(\varepsilon)$ along $c(t)$. That is, $c|_{[\varepsilon,t_\varepsilon]}$ is minimal. Hence,
\begin{align*}
2\,\tilde{\mathfrak{i}}_M&\leq 2\tilde{d}(p_\varepsilon,q_\varepsilon)=d(p_\varepsilon,q_\varepsilon)+d(q_\varepsilon,p_\varepsilon)\leq d(p_\varepsilon,q_\varepsilon)+d(q_\varepsilon,p)+d(p,p_\varepsilon)\tag{3.3}\label{new**2}\\
&=d(p_\varepsilon,q)+d(q,q_\varepsilon)+d(q_\varepsilon,p)+d(p,p_\varepsilon)\\
&=d(p,p_\varepsilon)+d(p_\varepsilon,q)+d(q,q_\varepsilon)+d(q_\varepsilon,p)\\
&=d(p,q)+d(q,p)=L(c)=2\,\tilde{\mathfrak{i}}_M,
\end{align*}
which implies that (see (\ref{new**2}))
\[
d(q_\varepsilon,p_\varepsilon)=d(q_\varepsilon,p)+d(p,p_\varepsilon)=L(c_3|_{[t_\varepsilon-d(p,q),d(q,p)]})+L(c_1|_{[0,\varepsilon]}).
\]
Hence, $c=c_1*c_3$ is smooth at $p$.
\end{proof}

In \cite{Sh1}, Shen introduces T-curvature, which is an important non-Riemannian quantity. However, the definition of the bound on T-curvature seems a little complicated. For convenience, we give a new definition of the bound on T-curvature. Also refer to \cite{Sh1,Z2} for more details.
\begin{definition}
Given $y,v\in T_xM$ with $y\neq0$, define the
T-curvature $\mathbf{T}$ as
\begin{equation*}
\mathbf{T}_y(v):=g_y(\nabla^V_vV,y)-g_y(\nabla_v^YV,y),
\end{equation*}
where $V$ (resp. $Y$) is a vector field with $V_x=v$ (resp. $Y_x=y$).
Set
\[
\mathbf{T}_p:=\sup_{y,v\in S_pM}|\mathbf{T}_{y}(v)|,\ \mathbf{T}_M:=\sup_{p\in M}\mathbf{T}_p.
\]
\end{definition}
Clearly, for a compact Finsler manifold, $\mathbf{T}_M$ is finite. And $\mathbf{T}_M=0$ if and only if $F$ is Berwalden.
By the proof of \cite[Theorem 1.1]{Z2}, we have the following result.
\begin{theorem}\label{closedgeodesci2}
Let $(M,F)$ be a compact Finsler $n$-manifold with $\mathbf{K}_M\geq k$, $\mathbf{T}_M \leq \tau$, $\Lambda_F\leq \Lambda$ and $\diam(M)\leq D$.
Then for any simple closed geodesic $\gamma$,
\[
L(\gamma)\geq\frac{\mu(M)}{c_{n-2}\Lambda^\frac{3n}{2}\left[\frac{\mathfrak{s}^{n-1}_k\left({\min\left\{D,\frac{\pi}{2\sqrt{k}}\right\}}\right)}{n-1}+\Lambda^{\frac12}\,\tau\int_0^{D}\mathfrak{s}^{n-1}_k(t) dt\right]},
\]
where $\mu(M)$ is either the Busemann-Hausdorff volume or the Holmes-Thompson volume of $M$ and $c_{n-2}:=\vol(\mathbb{S}^{n-2})$.
\end{theorem}
Theorem \ref{closedgeodesci} together with Theorem \ref{closedgeodesci2} then yields Theorem \ref{injectiveestimeat}.

\begin{remark}By Theorem \ref{closedgeodesci} and the standard arguments (see \cite{AM,P}), one can show the following result, which is an extension of  the results in \cite{K,R,R2}.

Let $(M,F)$ be an even-dimensional, compact Finsler manifold with $0<\mathbf{K}_M\leq k$.

\noindent (1) If $M$ is orientable, then
\[
\mathfrak{i}_M\geq \frac{\text{Conj}_M }{\lambda_F} \geq \frac{\pi}{\lambda_F\sqrt{k}}.
\]
In particular, if $F$ is reversible, then $\mathfrak{i}_M=\text{Conj}_M$.

\noindent (2) If $M$ is not orientable, then
\[
\mathfrak{i}_M\geq\frac{\pi}{\lambda_F(1+\lambda_F)\sqrt{k}}.
\]

\end{remark}

\section{The convex radius of a Berwald manifold}
Recall that a subset $A\subset M$ is called {\it strongly (geodesically) convex} if for any $p, q\in A$, there exists a geodesic $\gamma_{pq}$ such
that $\gamma_{pq}$ is the unique minimizer in $M$ from $p$ to $q$, and $\gamma_{pq}$ is  the only
geodesic contained  in $A$ from $p$ to $q$.

\begin{definition}Let $(M,F)$ be a forward complete Finsler manifold.
The convexity radius at a point $x\in M$ is defined by
\[
\text{Conv}_x:=\sup\{r>0:\,B^+_x(s)\text{ is strongly convex for any }s<r \}.
\]
And the convexity radius of $(M,F)$ is defined by $\text{Conv}(M,F):=\inf_{x\in M}\text{Conv}_x$.
\end{definition}

In \cite{Sh1}, Shen estimates convexity radii  in the reversible Finslerian case. Here, we give an estimate on the convexity radius of a Berwald manifold.
\begin{theorem}\label{Berwaldinjecitve}
Let $(M,F)$ be a forward complete Berwald manifold with $\mathbf{K}_M\leq k$, $\mathfrak{i}_M\geq \varsigma$ and $\lambda_F\leq \lambda$. Then
\[
\text{Conv}(M,F)\geq \min\left\{\frac{\pi}{2\sqrt{k}},\frac{\varsigma}{\lambda(1+{\lambda})} \right\}.
\]
\end{theorem}
\begin{proof} Choosing an arbitrary point $x\in M$ and any $r\in (0,\min\{\frac{\pi}{2\sqrt{k}},\frac{\varsigma}{\lambda(1+\lambda)} \})$, we now show that $B^+_x(r)$ is strictly convex.

For each two points $p_1,p_2\in B^+_x(r)$, let $\gamma_{p_1p_2}(t)$, $0\leq t\leq l$ denote a normal minimal geodesic from $p_1$ to $p_2$. Since
\[
l=L(\gamma_{p_1p_2})\leq d(p_1,x)+d(x,p_2)\leq \lambda\cdot d(x,p_1)+d(x,p_2)<{\varsigma}/{\lambda},
\]
 $\gamma_{p_1p_2}$ is the unique minimal geodesic from $p_1$ to $p_2$ and hence, $\rho(\cdot):=d(x,\cdot)$ is smooth on $\gamma_{p_1p_2}([0,l])-\{x\}$.

 Fix a point $p\in B^+_x(r)$ and set
\[
\text{Co}_p:=\{q\in B^+_x(r):\, \gamma_{pq}([0,l])\subset B^+_x(r)\}.
\]

We first prove that $\text{Co}_p$ is an open subset of $B^+_x(r)$. For any sequence $\{q_n\}\subset  B^+_x(r)-\text{Co}_p$ converging to some point $q\in B^+_x(r)$, there exists $t_n\in (0,l)$ such that $\rho(\gamma_{pq_n}(t_n))\geq r$  for each $n$.
Since $\{\gamma_{pq_n}\}$ is uniformly bounded, by the Arzel\'a-Ascoli theorem \cite{DYS}, we can assume  that $\{\gamma_{pq_n}\}$ converges to the minimal geodesic $\gamma_{pq}$ and $t_n\rightarrow t_0$. Clearly, $\rho(\gamma_{pq}(t_0))\geq r$, which implies $q\in B^+_x(r)-\text{Co}_p$. Hence, $B^+_x(r)-\text{Co}_p$ is a closed subset of $B^+_x(r)$.

Secondly, we claim that $\text{Co}_p$ is a closed subset of $B^+_x(r)$. It suffices to show $\partial \text{Co}_p\subset \text{Co}_p$. Given any point $q\in  \partial \text{Co}_p$, the argument is divided into the following two cases:

\noindent\textbf{Case 1.}
Suppose $x\notin\gamma_{pq}$. Then $\rho\circ\gamma_{pq}(t)$ is smooth, and the Hessian comparison theorem \cite{Sh1} implies that $\frac{d^2}{dt^2}\rho\circ\gamma_{pq}(t)\geq 0$, which implies that \[
\rho\circ\gamma_{pq}(t)\leq \max\{\rho\circ\gamma_{pq}(0),\rho\circ\gamma_{pq}(1)\}<r.
\]
Hence, $\gamma_{pq}([0,1])\subset B^+_p(r)$ and therefore, $q\in \text{Co}_p$.

\noindent\textbf{Case 2.} If there exists $t_0\in [0,l]$ such that $\gamma_{pq}(t_0)=x$, then for $t\in [t_0,l]$,
\[
\rho(\gamma_{pq}(t))=L(\gamma_{pq}|_{[t_0,t]})\leq L(\gamma_{pq}|_{[t_0,l]})=\rho(q)<r.
\]
On the other hand, there exists $s\in [0,t_0)$ such that $\rho(\gamma_{pq}(s))<r/{\lambda^2}$. Thus, for $t\in [s,t_0]$,
\begin{align*}
\rho(\gamma_{pq}(t))&\leq \lambda\cdot d(\gamma_{pq}(t),x)\leq\lambda\cdot d(\gamma_{pq}(s),x)\leq \lambda^2 \cdot\rho(\gamma_{pq}(s))<r.
\end{align*}
Note that $\gamma_{pq}(t)$, $t\in [0,s]$ is the unique minimal geodesic from $p$ to $\gamma_{pq}(s)$. Since $\rho(\gamma_{pq}(s))<r$, the argument of Case 1 implies that  $\gamma_{pq}([0,s])\subset B^+_p(r)$.
Hence, $\gamma_{pq}([0,1])\subset B^+_p(r)$ and therefore, $q\in \text{Co}_p$.

From above, we see that $\text{Co}_p$ is a both open and closed subset of $B^+_x(r)$. Since $x\in \text{Co}_p$, $\text{Co}_p=B^+_x(r)$ and hence, $B^+_x(r)$ is strictly convex.
\end{proof}

\begin{remark}\label{convexrad}Denote by $\mathbf{T}^s_M$ the upper bound of T-curvature in the sense of Shen \cite{Sh1}.
Using the argument above, one can obtain an estimate
on the convexity radius of a general Finsler manifold. More precisely,
let $(M,F)$ be a forward complete Finsler manifold with $\mathbf{K}_M\leq k$, $\mathfrak{i}_M\geq \varsigma$, $\lambda_F\leq \lambda$ and $\mathbf{T}^s_M\leq \xi$. Then
\[
\text{Conv}(M,F)\geq \min\left\{\mathfrak{v},\frac{\varsigma}{\lambda(1+{\lambda})} \right\},
\]
where $\mathfrak{v}$ is the first positive zero of the following equation
\[
{\mathfrak{s}'_k(t)}-\xi\cdot \mathfrak{s}_k(t)=0.
\]
This estimate coincides with Shen's result \cite[Theorem 15.2.1]{Sh1} in the reversible case.
\end{remark}

\begin{proposition}\label{impconvele}
 Let $(M,F)$ be a forward complete Berwald manifold with $\mathbf{K}_M\leq k$ and  $\mathfrak{i}_M\geq \varsigma$. Set $l:=\min\{\pi/(2\sqrt{k}),{\varsigma} \}$. Given any $x\in M$ and any $0<r< l$,
if a geodesic $\gamma$ is tangent to the forward sphere $S^+_x (r)=\partial B^+_x(r)$ at $q$, then there exists a small neighborhood $U_q$ of $q$ such that $U_q\cap \gamma$ is outside $\overline{B^+_x(r)}-\{q\}$.
\end{proposition}
\begin{proof}
Suppose that $\gamma$ is a normal geodesic.
Let $\rho(\cdot):=d(x,\cdot)$. Clearly, for any $p\in B^+_x(r)-\{x\}$, $\rho(p)$ is smooth.
Set $\gamma(t_0):=q$. Since $\nabla\rho$ is the normal vector field along $S^+_p(r)$, we have
\[
\left.\frac{d}{dt}\right|_{t=t_0}\rho(\gamma(t))=g_{\nabla\rho}\left(\nabla\rho,\dot{\gamma}(t_0)\right)=0.\tag{4.1}\label{newgausslemma}
\]
Clearly, $\rho\circ\gamma(t_0)=r< l$ implies that there is a small number $\epsilon>0$ such that $\rho\circ\gamma(t)<l$, for $t\in (t_0-\epsilon,t_0+\epsilon)$.
Thus, it follows from Hessian comparison theorem  that
$
\frac{d^2}{dt^2}\rho(\gamma(t))>0$ for $t\in (t_0-\epsilon,t_0+\epsilon)$.
This together with (\ref{newgausslemma}) yields that $\rho\circ\gamma$ has a minimum at $t_0$, which implies the conclusion.
\end{proof}

In the rest of this section, we assume that $(A,d\mathfrak{m})$ is a measure space of volume $1$, and $(M,F)$ is a forward complete Berwald $n$-manifold.
Given $p\in M$ and $r>0$,
any measurable map $f:A\rightarrow B^+_p(r)$ is called a {\it{mass distribution}} on $B^+_p(r)$. Define a vector filed $V$ on $\overline{B^+_p(r)}$ by
\[
V(x):=-\int_{A}  \exp_x^{-1} f(a) \,d\mathfrak{m}(a)            .
\]
Then we have the following theorem. Refer to \cite{Ka} for the results of the center of mass in the Riemannian case.
\begin{theorem}\label{centermass}Let $(M,F)$ be a forward complete Berwald $n$-manifold with $|\mathbf{K}_M|\leq k$, $\Lambda_F\leq \Lambda$, and $\mathfrak{i}_M\geq\varsigma$. There exists a constant $\mathfrak{r}=\mathfrak{r}(n,k,\Lambda,\varsigma)>0$ such that for each $0<r<\mathfrak{r}$, each $p\in M$ and each measurable map $f:A\rightarrow B^+_p(r)$, there exists a unique point $q\in B^+_p(r)$ with $V(q)=0$. $q$ is called the center of mass $\mathscr{C}(f)$ of $f$.

In particular, $V(q)$ is differentiable and the map $V_{*q}$ is non-degenerate at $q=\mathscr{C}(f)$, where $V_*:TB^+_p(r)\rightarrow TB^+_p(r)$ is defined by
\[
V_*(X)=X^i\frac{\partial V^k}{\partial x^i}\frac{\partial}{\partial x^k},\ \forall X=X^i\frac{\partial}{\partial x^i}
\]
and $(x^i,y^i)$ is a local coordinate system of $TB^+_p(r)$.
\end{theorem}
\begin{proof}Let $\mathfrak{r}:=\frac{1}{2\Lambda}\min\{\frac{\pi}{2\sqrt{k}},\frac{\varsigma}{1+\sqrt{\Lambda}},\mathfrak{t},\frac{1}{40\Lambda^2}\}$, where $\mathfrak{t}$ is as in  Lemma \ref{newJacobies}. Given $p\in M$ and $f$, we consider $V(x)$ defined on $\overline{B^+_p(r)}$.

\noindent \textbf{Step 1.}
First, we show that for each $x\in \partial {B^+_p(r)}$, $V(x)$ is a nonzero outward vector.
For each $a\in A$, set $X_a:=-\exp_{x}^{-1}f(a)$. It is easy to see that the geodesic $\gamma_{X_a}(t)$, $t\in (-\epsilon,0)$ is contained in  $B^+_p(r)$ and $\gamma_{X_a}(t)$, $t\in (0,\epsilon)$ is outside $B^+_p(r)$. Set $\rho(\cdot)=d(p,\cdot)$. Thus,
\[
\left.\frac{d}{dt}\right|_{t=0}\rho(\gamma_{X_a}(t))=g_{\nabla \rho(x)}(\nabla \rho(x),{X_a})\geq 0.
\]
If $g_{\nabla \rho(x)}(\nabla \rho(x),{X_a})=0$, Proposition \ref{impconvele} yields that $\gamma_{X_a}(t)$, $t\in (-\epsilon,0)\cup (0,\epsilon)$ is outside $\overline{B^+_p(r)}$, which is a contradiction. Hence, $g_{\nabla \rho(x)}(\nabla \rho(x),{X_a})>0$ and
\[
g_{\nabla \rho(x)}(\nabla \rho(x),V(x))=g_{\nabla \rho(x)}\left(\nabla \rho(x),\int_{A}{X_a}\,d\mathfrak{m}(a)\right)>0,
\]
which implies that $V(x)$ is a nonzero outward vector.

\noindent \textbf{Step 2.} Now we show that $V$ has only isolated singularities in $B^+_p(r)$. Given $a\in A$ and a geodesic $\gamma(s)$, $s\in [0,1]$ in $B^+_p(r)$, consider the geodesic variation
\[
\sigma_a(t,s)=\exp_{\gamma(s)}(t-1)X_a, \ t\in [0,1],
\]
where $X_a:=-\exp_{\gamma(s)}^{-1}f(a)$.
Clearly, $\sigma_a(0,s)=f(a)$ and $\sigma_a(s,1)=\gamma(s)$. Note that
\[
U_{s;a}(t)=\frac{\partial}{\partial s}\sigma_a(t,s)
\]
is a Jacobi field with
$U_{s;a}(0)=0$ and $U_{s;a}(1)=\dot{\gamma}(s)$. Set
\[
T_{s;a}(t):=\frac{\partial}{\partial t}\sigma_a(t,s)=(\exp_{\gamma(s)})_{*(t-1)X_a}X_a.
\]
Clearly,
\begin{align*}
U'_{s;a}(1) =D_{T_{s;a}}{U_{s;a}} =D_{U_{s;a}}T_{s;a}=D_{U_{s;a}}X_\alpha.
\end{align*}
Thus, Lemma \ref{newJacobies} together with the equalities above implies that
\[
\|\dot{\gamma}(s)-D_{U_{s;a}}X_\alpha\|_{\dot{\gamma}(s)}\leq \frac{1}{20}\|\dot{\gamma}(s)\|_{\dot{\gamma}(s)},
\]
where $\|\cdot\|_{\dot{\gamma}(s)}:=\sqrt{g_{_{\dot{\gamma}(s)}}(\cdot,\cdot)}$.
Since $F$ is Berwalden, we have
\begin{align*}
&\left\|\dot{\gamma}(s)- D_{\dot{\gamma}(s)}V\right\|_{\dot{\gamma}(s)}
=\left\|\int_A(\dot{\gamma}(s)-D_{U_{s;a}}X_\alpha)\,d\mathfrak{m}(a)\right\|_{\dot{\gamma}(s)}\\
\leq & \int_A\|\dot{\gamma}(s)-D_{U_{s;a}}X_\alpha\|_{\dot{\gamma}(s)}\,d\mathfrak{m}(a)\leq \int_A\frac{1}{20}\|\dot{\gamma}(s)\|_{\dot{\gamma}(s)}\,d\mathfrak{m}(a)=\frac{1}{20}\|\dot{\gamma}(s)\|_{\dot{\gamma}(s)},\tag{4.2}\label{lastjA}
\end{align*}
which implies that $V$ has only isolated singularities.

\noindent \textbf{Step 3.} We now show that $V(x)$ has exactly one singularity in $B^+_p(r)$.
Since $B^+_p(r)$ is contractible and $V$ is a outward vector field along the boundary, the sum of index of $V$ in $B^+_p(r)$ is $+1$, which implies that $V$ has at least one  isolated singularity in $B^+_p(r)$.

On the other hand,  for each isolated singularity $z$ in $B^+_p(r)$, let $\gamma(s)$ be a geodesic from $z$. (\ref{lastjA}) implies that
\[
\left.\frac{d}{ds}\right|_{s=0}g_{\dot{\gamma}(s)}(\dot{\gamma}(s),V(\gamma(s)))=g_{\dot{\gamma}(0)}(\dot{\gamma}(0),D_{\dot{\gamma}}V)=\varepsilon_0>0.
\]
Then there exists a small $s_0>0$ such that  for $s\in [0,s_0]$,
\[
\frac{d}{ds}g_{\dot{\gamma}(s)}(\dot{\gamma}(s),V(\gamma(s)))\geq \frac12\varepsilon_0\Rightarrow g_{\dot{\gamma}(s)}(\dot{\gamma}(s),V(\gamma(s)))\geq \frac12\varepsilon_0s.\tag{4.3}\label{outward}
\]
 We claim that there exists a small $l>0$ such that along $\partial B^+_{z}(l)$, $V$ is outward. If the claim is true, then the Poincar\'e-Hopf theorem implies that the index of $V$ at $z$ is $+1$ and therefore, $V$ has exactly one zero in $B^+_p(r)$.

 Suppose that the claim is not true. Let $\xi(\cdot):=d(z,\cdot)$. Then (\ref{outward}) yields that there exists a sequence $l_n\downarrow 0$ and a sequence $y_n\in S_zM$ such that for each $n$, there is a point $x_n=\exp_z(l_ny_n)\in \partial B^+_{z}(l_n)$ is the first point along $\gamma_{y_n}$ with $g_{\nabla\xi(x_n)}(\nabla\xi(x_n),V(x))= 0$. Since $\nabla\xi(x_n)=\dot{\gamma}_{y_n}(l_n)$, (\ref{outward}) implies that $l_n$ is the minimal point of $g_{\dot{\gamma}_{y_n}}(\dot{\gamma}_{y_n}(s),V(\gamma_{y_n}))$ and hence,
 \[
 \left.\frac{d}{ds}\right|_{s=l_n}g_{\dot{\gamma}_{y_n}}(\dot{\gamma}_{y_n}(s),V(\gamma_{y_n}))=g_{\dot{\gamma}_{y_n}(l_n)}(\dot{\gamma}_{y_n}(l_n),D_{\dot{\gamma}_{y_n}}V)\leq 0.\tag{4.4}\label{use1}
 \]
Since $S_zM$ is compact, we can assume that $y_n\rightarrow y_0\in S_zM$. Thus,
\[
\dot{\gamma}_{y_n}(l_n)=(\exp_z)_{*l_ny_n}y_n\rightarrow y_0.\tag{4.5}\label{use2}
\]
(\ref{use1}) together with (\ref{use2}) implies that
 \[
 \left.\frac{d}{ds}\right|_{s=0}g_{\dot{\gamma_{y_0}}(s)}(\dot{\gamma_{y_0}}(s),V(\gamma_{y_0}(s)))=g_{y_0}(y_0,D_{y_0}V)\leq0,
 \]
which is a contradiction. Therefore, the claim is true.

\noindent \textbf{Step 4.} From above, one can see that $V(x)$ is differentiable at every point $x\in B^+_p(r)$, and $D_{X}V\neq 0$  for any $X\in T_{\mathscr{C}(f)}M-\{0\}$. Let $(x^i,y^i)$ be a coordinate system of $TB^+_p(r)$ and let $\gamma(t)$, $t>0$ be a smooth curve from $\mathscr{C}(f)$ with $\dot{\gamma}(0)=X$. Thus,
\begin{align*}
0&\neq D_{X}V=\left[\frac{dV^i}{dt}+\Gamma^i_{jk}({{\gamma}(0)})\dot{\gamma}^j(0)V^k(\mathscr{C}(f))\right]\frac{\partial}{\partial x^i}\\
&=\left.\frac{\partial V^i}{\partial x^k}\right|_{\mathscr{C}(f)}\dot{\gamma}^k(0)\frac{\partial}{\partial x^i}=V_{*\mathscr{C}(f)}(X),
\end{align*}
which implies that $V_{*\mathscr{C}(f)}$ is nonsingular.
 \end{proof}

\section{A Cheeger type finiteness theorem for Berwald manifolds}
Given $n\in \mathbb{N}$, $\Lambda\geq 1$, $\varsigma>0$ and $k\geq 0$, let $\mathfrak{r}:=\mathfrak{r}(n,k,\Lambda,\varsigma)$ and $\mathfrak{C}:=\mathfrak{C}(n,k,\Lambda)$ be defined as in Theorem \ref{centermass} and Lemma \ref{parallelesi}, respectively.
\begin{definition}\label{condition 11}
We say a triple $(R,\varepsilon_1,\varepsilon_2)$ satisfies {{Condition ($\Delta$)}} if
\begin{align*}
(1)\  &0<R\leq\min\left\{\frac{\mathfrak{r}}{40\cdot\Lambda^4},\,\mathcal {C}_0,\, \mathcal {C}_1,\,\mathcal {C}_2\right\},\\
(2)\ &0<\varepsilon_1\leq \frac{R}{12\Lambda^3},\ \varepsilon_2>0,\\
(3)\ &\frac{(1-kR^2)}{\Lambda^5}-\mathcal {C}_{3}(n,k,\Lambda,R,\varepsilon_2)-\frac{2^{2n+6}\Lambda^{4n+6}}{\mathfrak{s}_k(\sqrt{\Lambda}R)}\varepsilon_1 >0,
\end{align*}
where
\begin{align*}
&\mathcal {C}_0(k,\Lambda):=\sup\left\{t>0:\frac{\mathfrak{s}_{-k}(3\Lambda^{\frac52}t)}{3\Lambda^{\frac52}t}\leq 2\right\},\\
&\mathcal {C}_1(n,k,\Lambda):=\sup\left\{t>0:\frac{\int^{\Lambda t}_0\mathfrak{s}_{-k}^{n-1}(s)ds}{\int^{\frac{t}{4\Lambda}}_0\mathfrak{s}_{k}^{n-1}(s)ds}\leq  2(4\Lambda^2)^n\right\},\\
&\mathcal {C}_2(k,\Lambda):=\sup\left\{t>0:\frac{t}{\mathfrak{s}_{-k}(t)}\frac{\mathfrak{s}_{k}(\Lambda^{\frac32}t)}{\mathfrak{s}_{-k}(\sqrt{\Lambda}t)}\geq 1-kt^2 \right\},\\
&\mathcal {C}_3(n,k,\Lambda,R,\varepsilon_2):=\frac{6\Lambda^3 R}{\mathfrak{s}_k(\sqrt{\Lambda}R)}\left( \frac{\mathfrak{s}_{-k}(\sqrt{\Lambda}R)}{\sqrt{\Lambda}R}-1 \right) \frac{\mathfrak{s}_{-k}(\sqrt{\Lambda}R)}{\mathfrak{s}_k(\sqrt{\Lambda}R)}+30 \Lambda^3 \mathfrak{C}(n,k,\Lambda) R^2+\Lambda \varepsilon_2.
\end{align*}
\end{definition}
In the following, we assume that $(R,\varepsilon_1,\varepsilon_2)$ is given and satisfies Condition ($\Delta$).

\begin{definition}\label{Condition1}
Given $N\in \mathbb{N}$, we say a compact Berwald $n$-manifold $(M,F)$ satisfies Condition (1-$N$) if

 (1) $\Lambda_F\leq \Lambda$, $|\mathbf{K}_M|\leq k$, $\mathfrak{i}_{M}\geq \varsigma,$ $\diam(M)\geq D$;

 (2) $M$ can be covered by $N$ convex balls of radius $R/(2\Lambda^\frac32)$
\[
\{ B^+_{p_\alpha}(R/(2{\Lambda}^\frac32)):\ \alpha=1,\ldots, N\}
\]
and such balls $\{B^+_{p_\alpha}(R/(4{\Lambda}^2))\}_{\alpha=1}^N$ are disjoint.
\end{definition}

Let $(M_i,F_i)$, $i=1,2$ be two Berwald $n$-manifolds satisfying Condition (1-$N$). Let $\{ B^+_{p^i_\alpha}(R/(2{\Lambda}^\frac32)):\ \alpha=1,\ldots, N\}$ be the forward convex balls of $M_i$ as in Definition \ref{Condition1}.

Let $(\mathbb{R}^n,\|\cdot\|)$ be a standard Euclidean space.
For each $i$, denote by $\|\cdot\|_i$ the average Riemannian norm on $M_i$ induced by $F_i$, which yields a linear isometry $u^i_\alpha:(\mathbb{R}^n,\|\cdot\|)\rightarrow (T_{p^i_\alpha}M_i,\|\cdot\|_{i})$ for each $\alpha\in\{1,\ldots,N\}$ such that for
\[
\frac{1}{\sqrt{\Lambda}}\leq \frac{F_i(u^i_\alpha(X))}{\|X\|}\leq \sqrt{\Lambda},\ \forall\,X\in \mathbb{R}^n.\tag{5.1}\label{u*}
\]
Hence, $u^i_\alpha:\overline{\mathcal {B}_0(R)}\rightarrow \overline{\mathcal {B}^+_{p^i_\alpha}(\sqrt{\Lambda}R)}$, where $\mathcal {B}_0(R)$ (resp. $\mathcal {B}^+_{p^i_\alpha}(R)$) denotes the ball of radius $R$ centered at the origin in $(\mathbb{R}^n,\|\cdot\|)$ (resp. $(T_{p^i_\alpha}M_i,F_i)$).

 Set
\[
\phi^i_\alpha:=\exp_{p_\alpha}\circ\, u^i_\alpha: \overline{\mathcal {B}_0(R)}\rightarrow \overline{{B}^+_{p^i_\alpha}(\sqrt{\Lambda}R)}.
\]
Clearly, $\phi^i_\alpha(\overline{\mathcal {B}_0(R)})\supset B^+_{p^i_\alpha}(R/\sqrt{\Lambda})$. In particular,
if $\phi^i_\alpha(\overline{\mathcal {B}_0(R)})\cap \phi^i_\beta(\overline{\mathcal {B}_0(R)})\neq\emptyset$, the triangle inequality then yields that
\[
\phi^i_\beta(\overline{\mathcal {B}_0(R)})\subset \phi^i_\alpha(\overline{\mathcal {B}_0(3{\Lambda}^2R)}),\ \phi^i_\alpha(\overline{\mathcal {B}_0(R)})\subset \phi^i_\beta(\overline{\mathcal {B}_0(3{\Lambda}^2R)}).
\]
Hence, for any $\alpha,\beta$ with $\phi^i_\alpha(\overline{\mathcal {B}_0(R)})\cap \phi^i_\beta(\overline{\mathcal {B}_0(R)})\neq\emptyset$, we can define a map
\[
\mathfrak{f}^i_{\beta\alpha}:=(\phi^i_\beta)^{-1}\circ \phi^i_\alpha:\overline{\mathcal {B}_0(R)}\rightarrow \overline{\mathcal {B}_0(3{\Lambda}^2R)}.
\]
The following lemma follows from  Lemma \ref{Jacbi} directly.
\begin{lemma}\label{C_1}There exits a constant $\mathscr{C}=\mathscr{C}(\varsigma,k,\Lambda)$ such that
for any $\alpha,\beta$ with $\phi^i_\alpha(\overline{\mathcal {B}_0(R)})\cap \phi^i_\beta(\overline{\mathcal {B}_0(R)})\neq\emptyset$, we have
\[
\|\mathfrak{f}^i_{\beta\alpha}\|_{C^1}\leq \mathscr{C}.
\]
\end{lemma}

For any two $\alpha,\beta$ with $\phi^i_\alpha(\overline{\mathcal {B}_0(R)})\cap \phi^i_\beta(\overline{\mathcal {B}_0(R)})\neq\emptyset$, there exists a unique minimal geodesic $\gamma^i_{\alpha\beta}(t)$, $0\leq t\leq 1$ from $p^i_\alpha$ to $p^i_\beta$. Let $P^i_{\alpha\beta}$ denote the parallel transformation along $\gamma^i_{\alpha\beta}$ from $T_{p^i_\alpha}M_i$ to $T_{p^i_\beta}M_i$. Define a linear isomorphism $\mathfrak{g}^i_{\beta\alpha} :\mathbb{R}^n\rightarrow \mathbb{R}^n$ by
\[
\mathfrak{g}^i_{\beta\alpha}:=(u^i_\beta)^{-1}P^i_{\alpha\beta}u^i_\alpha.
\]
Since $F_i$ is Berwalden, $F_i(P^i_{\alpha\beta}Y)=F_i(Y)$. Thus, one has the following result.
\begin{lemma}\label{C_2}For any $\alpha,\beta$ with $\phi^i_\alpha({\overline{\mathcal {B}_0(R)}})\cap \phi^i_\beta(\overline{\mathcal {B}_0(R))}\neq\emptyset$, we have
\[
\|\mathfrak{g}^i_{\beta\alpha}\|_0:=\sup_{X\neq0}\frac{\|\mathfrak{g}^i_{\beta\alpha}X\|}{\|X\|}\leq \Lambda.
\]
\end{lemma}

By the Arzel\'a-Ascoli theorem, one can easily the following lemma.
\begin{lemma}\label{totallboun}Let $\mathscr{C}$ be as in Lemma \ref{C_1} and let $(\mathbb{R}^n,\|\cdot\|)$ be a Euclidean space. Set
\begin{align*}
&H_1:=\{f:\overline{\mathcal {B}_0(R)}\rightarrow \overline{\mathcal {B}_0(3{\Lambda}^2R)}:\, f\text{ is a embedding map with } \|f\|_{C_1}\leq \mathscr{C}\},\\
&H_2:=\{f:(\mathbb{R}^n,\|\cdot\|)\rightarrow (\mathbb{R}^n,\|\cdot\|):\, f\text{ is a linear map with } \|f\|_0\leq \Lambda\}.
\end{align*}
Then $H_1$ and $H_2$ are totally bounded. That is, for each $\varepsilon>0$, $H_i$ can be covered by a finite number of balls of radius $\varepsilon$.
\end{lemma}

\begin{definition}Given $N\in \mathbb{N}$,
we say two compact Berwald $n$-manifolds $(M_i,F_i)$, $i=1,2$ satisfy Condition (2-$N$) if

(1) $(M_i,F_i)$, $i=1,2$ satisfy Condition (1-$N$);

(2)
$\phi^1_\alpha(\overline{\mathcal {B}_0(R)})\cap \phi^1_\beta(\overline{\mathcal {B}_0(R)})\neq\emptyset\Leftrightarrow \phi^2_\alpha(\overline{\mathcal {B}_0(R)})\cap \phi^2_\beta(\overline{\mathcal {B}_0(R)})\neq\emptyset$, for all $\alpha,\beta\in \{1,\ldots,N\}$.
\end{definition}

The following result is a generalized Peter's lemma, which will be proved in next section. Also refer to \cite{Pe} for Peter's lemma in the Riemannian case.
\begin{lemma}\label{Peter} Let $(R,\varepsilon_1,\varepsilon_2)$ be a triple satisfying Condition ($\Delta$)
and let $(M_i,F_i)$, $i=1,2$ be two closed Berwald manifolds satisfying Condition (2-$N$).
Suppose that  for any $\alpha,\beta\in \{1,\ldots,N\}$ with $\phi^i_\alpha(\overline{\mathcal {B}_0(R)})\cap \phi^i_\beta(\overline{\mathcal {B}_0(R)})\neq\emptyset$, we have
\begin{align*}
&\|\mathfrak{f}^1_{\beta\alpha}-\mathfrak{f}^2_{\beta\alpha}\|_{C_1}\leq \varepsilon_1,\\
&\|\mathfrak{g}^1_{\beta\alpha}-\mathfrak{g}^2_{\beta\alpha}\|_{0}\leq \varepsilon_2.
\end{align*}
Then $M_1$ and $M_2$ are diffeomorphic.
\end{lemma}

By Lemma \ref{Peter}, we now show the following theorem.
\begin{theorem}Given $n\in\mathbb{N}$, $\Lambda\geq 1$, $k\geq 0$ and $V,D>0$, there exist only finitely many diffeomorphism classes of compact Berwald
$n$-manifolds $(M,F)$ satisfying
\[
\Lambda_F\leq \Lambda,\ |\mathbf{K}_M|\leq k,\  \mu(M)\geq V,\ \diam(M)\leq D.\tag{5.2}\label{5.2**}
\]
where $\mu(M)$ is either the Busemann-Hausdorff volume or the Holmes-Thompson volume of $M$.
\end{theorem}
\begin{proof}
Theorem \ref{injectiveestimeat} yields a positive constant $\varsigma=\varsigma(n,\Lambda,k,V,D)$ such that if a compact Berwald $n$-Finsler manifold $(M,F)$ satisfies (\ref{5.2**}), then $\mathfrak{i}_M\geq \varsigma$. Let $(R,\varepsilon_1,\varepsilon_2)$ be a triple defined as in Definition \ref{condition 11}, i.e., $(R,\varepsilon_1,\varepsilon_2)$ satisfies Condition ($\Delta$).

Suppose the theorem is not true. Then there exists a infinite sequence $\{(M_s,F_s)\}$ satisfying (\ref{5.2**}),  but $\{(M_s,F_s)\}$ are not diffeomorphic mutually.

For each $s$, let $\{B^+_{p^s_\alpha}(R/(4\Lambda^2))\}_{\alpha=1}^{N_s}$ denote the maximal family of disjoint balls of radius $R/(4\Lambda^2)$  in $M_s$.
The volume comparison theorem \cite{ZY} then implies
\[
N_s\leq \frac{\mu(M_s)}{\min_{\alpha}\mu(B^+_{p^s_\alpha}(R/(4\Lambda^2)))}\leq C_0(n,\Lambda,D,R,k)=:N_0.
\]
It is not hard to check that $\{B^+_{p^s_\alpha}(R/(2{\Lambda^\frac32}))\}_{\alpha=1}^{N_s}$ cover $M_s$.
Since $\{(M_s,F_s)\}$ is a infinite sequence and $N_0$ is a finite number, there must be a subsequence $\{(M_{s_L},F_{s_L})\}$ such that all $N_{s_L}\equiv N_1\leq N_0$. That is, for all $L$, the number of the maximal family of disjoint balls of radius $R/(4\Lambda^2)$ in $M_{s_L}$ are the same. In particular, for each $L$, $M_{s_L}$ can be covered by $N_1$ balls of radius $R/(2{\Lambda^\frac32})$. Hence, all the elements in $\{(M_{s_L},F_{s_L})\}$ satisfy Condition (1-$N_1$).

 Since $N_1$ is finite, there must be a subsequence $\{(M_K,F_K)\}$ of $\{(M_{s_L},F_{s_L})\}$ such that for any $(M_{K_1}, F_{K_1}),(M_{K_2}, F_{K_2})\in \{(M_K,F_K)\}$,
\[
\phi^{K_1}_\alpha(\overline{\mathcal {B}_0(R)})\cap \phi^{K_1}_\beta(\overline{\mathcal {B}_0(R)})\neq\emptyset\Leftrightarrow \phi^{K_2}_\alpha(\overline{\mathcal {B}_0(R)})\cap \phi^{K_2}_\beta(\overline{\mathcal {B}_0(R)})\neq\emptyset,
\]
for all $\alpha,\beta\in \{1,\ldots,N_1\}$. That is, all the elements in $\{(M_K,F_K)\}$ satisfies Condition (2-$N_1$).

Lemma \ref{totallboun} yields that $H_i$ can be covered by a finite number (say $A_i$) of $\varepsilon_i$-balls, $i=1,2$. Hence, for each $K$, $\mathfrak{f}^K_{\beta\alpha}=(\phi^K_\beta)^{-1}\circ \phi^K_\alpha\in H_1$ is in some a $\varepsilon_1$-ball. Since $N_1,A_1$ are finite and $\{(M_K,F_K)\}$ is a infinite sequence, there exists a subsequence $\{(M_{K'},F_{K'})\}$ such that for any $(M_{K'_1}, F_{K'_1}), (M_{K'_2}, F_{K'_2})\in \{(M_{K'},F_{K'})\}$, we have
\[
\|\mathfrak{f}^{K'_1}_{\beta\alpha}-\mathfrak{f}^{K'_2}_{\beta\alpha}\|_{C_1}\leq \varepsilon_1,
\]
for any $\alpha,\beta\in \{1,\ldots,N_1\}$ with $\phi^{K'_1}_\alpha(\overline{\mathcal {B}_0(R)})\cap \phi^{K'_1}_\beta(\overline{\mathcal {B}_0(R)})\neq\emptyset$.

Likewise, since $N_1,A_2$ are finite and $\{(M_{K'},F_{K'})\}$ is infinite sequence,
there exists a subsequence $\{(M_{K''},F_{K''})\}$ such that for any $(M_{K''_1}, F_{K''_1}), (M_{K''_2}, F_{K''_2})\in \{(M_{K''},F_{K''})\}$, we have
\begin{align*}
\|\mathfrak{f}^{K''_1}_{\beta\alpha}-\mathfrak{f}^{K''_2}_{\beta\alpha}\|_{C_1}\leq \varepsilon_1, \ \|\mathfrak{g}^{K''_1}_{\beta\alpha}-\mathfrak{g}^{K''_2}_{\beta\alpha}\|_0\leq \varepsilon_2,
\end{align*}
for any $\alpha,\beta\in \{1,\ldots,N_1\}$ with $\phi^{K''_1}_\alpha(\overline{\mathcal {B}_0(R)})\cap \phi^{K''_1}_\beta(\overline{\mathcal {B}_0(R)})\neq\emptyset$.

 Lemma \ref{Peter} then implies that  $\{(M_{K''},F_{K''})\}$  are diffeomorphic mutually, which contradicts the definition of $\{(M_s,F_s)\}$.
\end{proof}

\section{A generalized Peter's Lemma}
We now recall some notations used in Sect. 5:

$\|\cdot\|$  denotes a Euclidean norm on $\mathbb{R}^n$ and  $\|\cdot\|_i:=\sqrt{\tilde{g}_i(\cdot,\cdot)}$ denotes the average Riemannian norm induced by $F_i$. In particular, for each $\alpha$, $u^i_\alpha:(\mathbb{R},\|\cdot\|)\rightarrow (T_{p^i_\alpha}M_i,\|\cdot\|_i)$ is a natural isometry. Given $X\in TM_i-\{0\}$, $\|\cdot\|_X:=\sqrt{{g_i}_X(\cdot,\cdot)}$, where $g_i$ is the fundamental tensor induced by $F_i$.
\begin{lemma} Let $(R,\varepsilon_1,\varepsilon_2)$ be the triple satisfying Condition ($\Delta$)
and let $(M_i,F_i)$, $i=1,2$ be two compact Berwald manifolds satisfying Condition (2-$N$).
Suppose that  for any $\alpha,\beta\in \{1,\ldots,N\}$ with $\phi^i_\alpha(\overline{\mathcal {B}_0(R)})\cap \phi^i_\beta(\overline{\mathcal {B}_0(R)})\neq\emptyset$, we have
\begin{align*}
&\|\mathfrak{f}^1_{\beta\alpha}-\mathfrak{f}^2_{\beta\alpha}\|_{C_1}\leq \varepsilon_1,\\
&\|\mathfrak{g}^1_{\beta\alpha}-\mathfrak{g}^2_{\beta\alpha}\|_{0}\leq \varepsilon_2.
\end{align*}
Then $M_1$ and $M_2$ are diffeomorphic.
\end{lemma}
\begin{proof}
\textbf{Step 1.}
For each $\alpha$, define a map $\mathscr{F}_\alpha:=\phi^2_\alpha\circ (\phi_\alpha^1)^{-1}:\phi_\alpha^1(\overline{\mathcal {B}_0(R)})\rightarrow \phi_\alpha^2(\overline{\mathcal {B}_0(R)})$.
Given $p\in \phi^1_\alpha(\overline{\mathcal {B}_0(R)})\cap \phi^1_\beta(\overline{\mathcal {B}_0(R)})$,  we now estimate $d(\mathscr{F}_\alpha(p),\mathscr{F}_\beta(p))$.
Since $(\phi^1_\alpha)^{-1}(p)\in \overline{\mathcal {B}_0(R)}$,
\[
\|\mathfrak{f}^1_{\beta\alpha}\circ(\phi^1_\alpha)^{-1}(p)-\mathfrak{f}^2_{\beta\alpha}\circ(\phi^1_\alpha)^{-1}(p)\|_{C_1}\leq \varepsilon_1.\tag{6.1}\label{5.2}
\]
Note that

\[
\mathfrak{f}^1_{\beta\alpha}\circ(\phi^1_\alpha)^{-1}(p)=(\phi^2_\beta)^{-1}\mathscr{F}_\beta(p),\
\mathfrak{f}^2_{\beta\alpha}\circ(\phi^1_\alpha)^{-1}(p)=(\phi^2_\beta)^{-1}\circ \mathscr{F}_\alpha(p).
\]
Hence, (\ref{5.2}) implies that
\begin{align*}
&F\left(\exp^{-1}_{p^2_{\beta}}(\mathscr{F}_\beta(p))-\exp^{-1}_{p^2_{\beta}}(\mathscr{F}_\alpha(p))\right)\\
\leq &\sqrt{\Lambda}\left\|u^{-1}_\beta\circ\exp^{-1}_{p^2_{\beta}}(\mathscr{F}_\beta(p))-u^{-1}_\beta\circ\exp^{-1}_{p^2_{\beta}}(\mathscr{F}_\alpha(p))\right\|\\
=&\sqrt{\Lambda}\left\|(\phi^2_\beta)^{-1}\mathscr{F}_\beta(p)-(\phi^2_\beta)^{-1}\circ \mathscr{F}_\alpha(p)\right\|\leq \sqrt{\Lambda}\,\varepsilon_1.\tag{6.2}\label{5.3}
\end{align*}
Clearly, $\mathscr{F}_\beta(p)\in \phi^2_\beta(\overline{\mathcal {B}_0(R)})$, that is, $d(p^2_{\beta}, \mathscr{F}_\beta(p))\leq\sqrt{\Lambda} R$.
Since $\phi^1_\alpha(\overline{\mathcal {B}_0(R)})\cap \phi^1_\beta(\overline{\mathcal {B}_0(R)})\neq\emptyset$, we have $
\phi^2_\alpha(\overline{\mathcal {B}_0(R)})\cap \phi^2_\beta(\overline{\mathcal {B}_0(R)})\neq\emptyset$ and $\mathscr{F}_\alpha(p)\in \phi^2_\beta(\overline{\mathcal {B}_0(3\Lambda^2R)})$.

Let $\gamma(t)$, $t\in [0,1]$ be a curve from $\mathscr{F}_\alpha(p)$ to $\mathscr{F}_\beta(p)$ with $\exp^{-1}_{p^2_\beta}(\gamma(t))$ is a straight line. Clearly, $\|\exp^{-1}_{p^2_\beta}(\gamma(t))\|_2\leq 3\Lambda^2R$,
which implies that $\gamma(t)\in \overline{B^+_{p^2_\beta}(3\Lambda^{5/2}R)}$ and
\[
\max_{t\in [0,1]}\frac{\mathfrak{s}_{-k}(d(p^2_\beta,\gamma(t)))}{d(p^2_\beta,\gamma(t))}\leq \frac{\mathfrak{s}_{-k}(3\Lambda^{5/2} R)}{3\Lambda^{5/2} R}.
\]
Now Lemma \ref{firstjab} and (\ref{5.3}) implies that
\[
d(\mathscr{F}_\alpha(p),\mathscr{F}_\beta(p))\leq \frac{\mathfrak{s}_{-k}(3\Lambda^{\frac52}R)}{3\Lambda R}\varepsilon_1<3{\Lambda}^\frac32\varepsilon_1.
\]

\noindent\textbf{Step 2.}
Let $\eta:\mathbb{R}^+\rightarrow [0,1]$ be a smooth function with $|\eta'|\leq 4$ and
\[
\eta(r)=\left \{
\begin{array}{lll}
&1,&0\leq r\leq \frac12,\\
&(0,1),&\frac12<r<1,\\
 &0, &r\geq 1.
\end{array}
\right.
\]
Given $p\in M_1$, set
\[
\eta_\alpha(p):=\eta\left(\frac{\|(\phi^1_\alpha)^{-1}(p)\|}{R} \right),\ \psi_\alpha(p):=\frac{\eta_\alpha(p)}{\sum_\alpha \eta_\alpha(p)}.
\]
Hence, $\eta_\alpha(p)>0$ (or $\psi_\alpha(p)>0$) if and only if $p\in \phi^1_{\alpha}(\mathcal {B}_0(R))$.

Given $p\in M_1$, we define a vector field on $M_2$ by
\[
V_p(x):=\sum_{\alpha=1}^N\psi_\alpha(p)\cdot \left(\exp^{-1}_{x} \mathscr{F}_\alpha(p)\right)=\sum_{\alpha\in N'_p}\psi_\alpha(p)\cdot \left(\exp^{-1}_{x} \mathscr{F}_\alpha(p)\right),
\]
where $N'_p:=\{\alpha:\psi_\alpha(p)\neq0\}=\{\alpha:p\in \phi^1_{\alpha}({\mathcal {B}_0(R)})\}$. Clearly, if $\alpha,\beta\in N'_p$, we have $\phi^1_{\alpha}(\overline{\mathcal {B}_0(R)})\cap \phi^1_{\beta}(\overline{\mathcal {B}_0(R)})\neq\emptyset$.
By Step 1, we have
\[
d(\mathscr{F}_\alpha(p),\mathscr{F}_\beta(p))<3{\Lambda^\frac32}\varepsilon_1,\ d(\mathscr{F}_\beta(p),\mathscr{F}_\alpha(p))<3\Lambda^\frac32\varepsilon_1.
\]
Hence, one can find a forward ball of radius $3{\Lambda^\frac32}\varepsilon_1$, say $B_2(3\Lambda^\frac32\varepsilon_1)$, such that
$\mathscr{F}_\alpha(p)\in B_2(3\Lambda^\frac32\varepsilon_1)$ for all $\alpha\in N'_p$.
Define a mass distribution $f_p:N'_p\rightarrow B_2(3\Lambda^\frac32\varepsilon_1)$ by $f_p(\alpha):=\mathscr{F}_\alpha(p)$. The measure $\mathfrak{m}_p$ on $N'_p$ is defined by $\mathfrak{m}_p(\alpha)=\psi_\alpha(p)$. Then
\[
V_p(x)=\int_{\alpha\in N'_p}\exp^{-1}_x(f_p(\alpha))\,d\mathfrak{m}_p(\alpha).\tag{6.3}\label{5.4}
\]
It follows from Theorem \ref{centermass} that there exists a unique $x_p\in B_2(3\Lambda^\frac32\varepsilon_1)$ such that $V_p(x_p)=0$. Now we define a map $\mathscr{F}:M_1\rightarrow M_2$ by $\mathscr{F}(p)=x_p$. It is easy to see that
$\mathscr{F}$ is well-defined, i.e.,
$x_p$ is independent of the choices of $B_2(3\Lambda^\frac32\varepsilon_1)$.

Set
\[
\mathcal {V}(p,x):=\left(\sum_{\alpha=1}^N \eta_\alpha(p)\right)\cdot V_p(x)=\sum_{\alpha=1}^N\eta_\alpha(p)\cdot \left(\exp^{-1}_{x} \mathscr{F}_\alpha(p)\right),
\]
Note that $\mathcal {V}(p,x)$ is $C^1$ (cf. Theorem \ref{centermass}).
Clearly, $\mathcal {V}(p,\mathscr{F}(p))=0$. The implicit function theorem then yields that
\[
[d\mathscr{F}]=-(D_2\mathcal {V})^{-1}\cdot D_1\mathcal {V},\tag{6.4}\label{5.5}
\]
where $D_i\mathcal {V}$ denotes the differential matrix of $\mathcal {V}$ respect to the $i$-th variable, i.e.,
\[
D_1\mathcal {V}:=\left(\frac{\partial \mathcal {V}}{\partial p}\right),\ D_2\mathcal {V}:=\left(\frac{\partial \mathcal {V}}{\partial x}\right).
\]
Note that
\[
D_2\mathcal {V}=\left(\sum_{\alpha=1}^N \eta_\alpha(p)\right)\cdot\left(\frac{\partial V_p}{\partial x} \right).
\]
Theorem \ref{centermass} then implies that $D_2\mathcal {V}$ is not singular at $x=\mathscr{F}(p)$. Hence, (\ref{5.5}) is well-defined.

We will show that $\mathscr{F}$ is an imbedding. (\ref{5.5}) implies that it is equivalent to show that $D_1\mathcal {V}|_{(p,\mathscr{F}(p))}$ is not singular.
Note that $\mathcal {V}(p,x)\in T_xM_2$ for fixed $x$. Let $\gamma(t)$, $t\in (-\epsilon,\epsilon)$ be a smooth curve with $\gamma(0)=p$ and $\dot{\gamma}(0)=X$. Thus,
\[
\left.\frac{d}{dt}\right|_{t=0}\mathcal {V}(\gamma(t),x)=\mathcal {D}_1\mathcal {V}|_{(p,x)}(X)\in T_{\mathcal {V}(p,x)}(T_{x}M_2)\cong T_{x}M_2.
\]
In the following, \textbf{we always set $x=\mathscr{F}(p)$}.
Clearly, $D_1\mathcal {V}|_{(p,x)}$ is not singular if and only if
\begin{align*}
0&\neq \left.\frac{d}{dt}\right|_{t=0}\mathcal {V}(\gamma(t),x)\\
&=\sum_{\alpha=1}^N\left[ \left.\frac{d\eta_\alpha(\gamma(t))}{dt}\right|_{t=0}\cdot Y_\alpha(p)+ \eta_\alpha(p)\cdot\left(\exp_x^{-1}\right)_{*\mathscr{F}_\alpha(p)}\left.\frac{d \mathscr{F}_\alpha(\gamma(t))}{dt}\right|_{t=0}\right]\\
&=\sum_{\alpha=1}^N\left[\langle X,d\eta_\alpha|_p\rangle\cdot Y_\alpha(p)+\eta_\alpha(p)\cdot\langle\langle X, dY_\alpha\rangle\rangle|_{p}\right],\tag{6.5}\label{5.6}
\end{align*}
where
\[
Y_\alpha(p):=\exp^{-1}_{x}\mathscr{F}_\alpha(p)\in T_xM_2,\ \langle\langle X, dY_\alpha\rangle\rangle|_{p}:=\left(\exp_x^{-1}\right)_{*\mathscr{F}_\alpha(p)}\left.\frac{d \mathscr{F}_\alpha(\gamma(t))}{dt}\right|_{t=0}.
\]
Now we show (\ref{5.6}).

\noindent \textbf{Step 3.} First, we now estimate
\[
\text{I}:=\sum_{\alpha=1}^N \langle X,d\eta_\alpha|_p\rangle\cdot Y_\alpha(p).
\]
Note that $d\eta_\alpha|_p\neq0$ if and only if
\[
\frac{R}{2}<\|(\phi^1_\alpha)^{-1}(p)\| <R\Leftrightarrow p\in \phi^1_\alpha(\mathcal {B}_0(R))-\phi^1_\alpha(\overline{\mathcal {B}_0(R/2)}).
\]
Thus,
\[
\text{I}=\sum_{\alpha\in N''_p} \langle X,d\eta_\alpha|_p\rangle\cdot Y_\alpha(p).
\]
where
\[
N''_p:=\left\{\alpha:\,p\in \phi^1_\alpha(\mathcal {B}_0(R))-\phi^1_\alpha(\overline{\mathcal {B}_0(R/2)})\right\}\subset N'_p.
\]
Recall that $\{B^+_{p^1_\alpha}(R/(4\Lambda))\}_{\alpha=1}^N$ are disjoint. Thus, we have
\[
\sharp N''_p\leq \frac{\mu (B^+_p(\Lambda R))}{\min_{\alpha\in N''_p}\mu(B^+_{p^1_\alpha}(R/(4\Lambda)))}\leq \Lambda^{2n}\frac{\int^{\Lambda R}_0\mathfrak{s}_{-k}^{n-1}(t)dt}{\int^{\frac{R}{4\Lambda}}_0\mathfrak{s}_{k}^{n-1}(t)dt}\leq  2^{2n+1}\Lambda^{4n}.\tag{6.6}\label{5.6''}
\]

For each $\alpha\in N''_p$, set $Z(t):=\exp^{-1}_{p^1_\alpha}(\gamma(t))\in T_{p^1_\alpha}M_1$.  Clearly, $F_1(Z(0))=d(p^1_\alpha,p)\in (\frac{R}{2\sqrt{\Lambda}},\sqrt{\Lambda}R)$ and $X=\left(\exp_{p^1_\alpha}\right)_{*d(p^1_\alpha,p)\nabla d(p^1_\alpha,p)}\dot{Z}(0)$. Hence, it follows from Lemma \ref{Jacbi} that
\[
F_1(\dot{Z}(0))\leq \Lambda\frac{d(p^1_\alpha,p)}{\mathfrak{s}_k(d(p^1_\alpha,p))}F_1(X)\leq \frac{\Lambda^{\frac32}R}{\mathfrak{s}_k(\sqrt{\Lambda}R)}F_1(X),
\]
which implies that
\begin{align*}
|\langle X,d\eta_\alpha|_p\rangle|
&\leq \frac{4}{R}\left|\left.\frac{d}{dt}\right|_{t=0}\left\|Z(t)\right\|_1\right|=\frac{4}{R}\left|\frac{\left.\frac{d}{dt}\right|_{t=0}\|Z(t)\|_1^2}{2\|Z(0)\|_1}\right|\\
&\leq \frac{4}{R}\frac{\|Z(0)\|_1\|\dot{Z}(0)\|_1}{\|Z(0)\|_1}\leq \frac{4\Lambda^2 }{\mathfrak{s}_k(\sqrt{\Lambda}R)}F_1(X)\tag{6.7}\label{5.6'}.
\end{align*}
Since $\alpha\in N''_p\subset N'_p$, Step 2 yields
\[
F_2(Y_\alpha(p))=d(x,\mathscr{F}_\alpha(p))<6\Lambda^2\varepsilon_1,
\]
which together with (\ref{5.6''}) and (\ref{5.6'}) implies that
\[
F_2(\text{I})= \sum_{\alpha\in N''_p} F_2(\langle X,d\eta_\alpha|_p\rangle\cdot Y_\alpha(p))\leq \frac{2^{2n+6}\Lambda^{4n+5}\varepsilon_1}{\mathfrak{s}_k(\sqrt{\Lambda}R)}F_1(X).
\]

\noindent \textbf{Step 4.} We now estimate
\[
\text{II}:=\sum_{\alpha=1}^N\eta_\alpha(p)\cdot\langle\langle X, dY_\alpha|_{p}\rangle\rangle=\sum_{\alpha\in N'_p}\eta_\alpha(p)\cdot\langle\langle X, dY_\alpha|_{p}\rangle\rangle.
\]
Given $\alpha\in N'_p$.
Since $(u^2_\alpha\circ (u^1_\alpha)^{-1})_*=u^2_\alpha\circ (u^1_\alpha)^{-1}$, for each $Z\in TM_1$, we have
\begin{align*}
F_2((u^2_\alpha\circ (u^1_\alpha)^{-1})_*Z)=F_2((u^2_\alpha\circ (u^1_\alpha)^{-1})Z)\geq \frac1{\sqrt{\Lambda}}\,\|(u^1_\alpha)^{-1}Z\|\geq \frac1{\Lambda} \,F_1(Z).\tag{6.8}\label{5.10}
\end{align*}
Recall that
\[
Y_\alpha(p)=\exp^{-1}_x\mathscr{F}_\alpha(p)
=\exp^{-1}_x\circ\exp_{p^2_\alpha}\circ u^2_\alpha\circ (u^1_\alpha)^{-1}\circ\exp^{-1}_{p^1_\alpha}(p).
\]
Thus, (\ref{5.10}) together with Lemma A.1 yields that
\begin{align*}
&F_2(\langle\langle X, dY_\alpha|_{p}\rangle\rangle)=F_2\left(\exp^{-1}_{x*}\circ\exp_{p^2_\alpha*}\circ (u^2_\alpha\circ (u^1_\alpha)^{-1})_*\circ\exp^{-1}_{p^1_\alpha*}X \right)\\
\geq& \frac{d(x,\mathscr{F}_\alpha(p))}{\Lambda\mathfrak{s}_{-k}(d(x,\mathscr{F}_\alpha(p)))}F_2\left(\exp_{p^2_\alpha*}\circ (u^2_\alpha\circ (u^1_\alpha)^{-1})_*\circ\exp^{-1}_{p^1_\alpha*}X \right)\\
\geq &\frac{d(x,\mathscr{F}_\alpha(p))}{\Lambda^4\mathfrak{s}_{-k}(d(x,\mathscr{F}_\alpha(p)))}\frac{\mathfrak{s}_k(F_2(u^2_\alpha\circ (u^1_\alpha)^{-1}\circ\exp^{-1}_{p^1_\alpha}(p)))}{F_2(u^2_\alpha\circ (u^1_\alpha)^{-1}\circ\exp^{-1}_{p^1_\alpha}(p))}\frac{d(p^1_\alpha,p)}{\mathfrak{s}_{-k}(d(p^1_\alpha,p))}F_1(X)\\
\geq& \frac{1}{\Lambda^5}\frac{R}{\mathfrak{s}_{-k}(R)}\frac{\mathfrak{s}_k(\Lambda^{\frac32}R)}{\mathfrak{s}_{-k}(\sqrt{\Lambda}R)}F_1(X)
\geq \frac{1}{\Lambda^5}(1-k R^2)F_1(X).\tag{6.9}\label{5.12}
\end{align*}

Since $\{B^+_{p^1_\alpha}(R/(2{\Lambda^\frac32}))\}_{\alpha=1}^N$ is a covering, there exists $\beta$ such that $d(p^1_\beta,p)< R/(2\sqrt{\Lambda})$, which implies that $\|(\phi^1_\beta)^{-1}(p)\|<R/2$. Hence, $\beta\in N'_p$ and $\sum_{\alpha\in N'_p}\eta_\alpha(p)\geq\eta_\beta(p)= 1$. We claim that
\[
F_2\left(\langle\langle X, dY_\beta|_{p}\rangle\rangle\right)-\sup_{\alpha\in N'_p}F_2\left( \langle\langle X, dY_\beta|_{p}-dY_\alpha|_{p}\rangle\rangle \right)>0,\tag{6.10}\label{positv}
\]
which will be proved in Step 5. Here,
\[
F_2\left( \langle\langle X, dY_\beta|_{p}-dY_\alpha|_{p}\rangle \rangle \right):=F_2\left( \langle\langle X, dY_\beta|_{p}\rangle \rangle -\langle\langle X, dY_\alpha|_{p}\rangle \rangle \right).
\]

By (\ref{5.12}) and (\ref{positv}), we have
\begin{align*}
F_2(\text{II})&=F_2\left(\sum_{\alpha\in N'_p}\eta_\alpha(p)\cdot\langle\langle X, dY_\alpha|_{p}\rangle\rangle\right)\\
&\geq\sum_{\alpha\in N'_p}\eta_\alpha(p)\cdot F_2\left(\langle\langle X, dY_\beta|_{p}\rangle\rangle\right)-\sum_{\alpha\in N'_p}\eta_\alpha(p)\cdot F_2\left( \langle\langle X, dY_\beta|_{p}-dY_\alpha|_{p}\rangle\rangle \right)\\
&\geq \left(\sum_{\alpha\in N'_p}\eta_\alpha(p)\right)\left[F_2\left(\langle\langle X, dY_\beta|_{p}\rangle\rangle\right)-\sup_{\alpha\in N'_p}F_2\left( \langle\langle X, dY_\beta|_{p}-dY_\alpha|_{p}\rangle\rangle \right)\right]\\
&\geq F_2\left(\langle\langle X, dY_\beta|_{p}\rangle\rangle\right)-\sup_{\alpha\in N'_p}F_2\left( \langle\langle X, dY_\beta|_{p}-dY_\alpha|_{p}\rangle \rangle \right)\\
&\geq \frac{(1-k R^2)}{\Lambda^5}  F_1(X)-\sup_{\alpha\in N'_p}F_2\left( \langle\langle X, dY_\beta|_{p}-dY_\alpha|_{p}\rangle \rangle \right),\tag{6.11}\label{5.13}
\end{align*}
In the following steps, we will show (\ref{positv}) and estimate (\ref{5.13}).

\noindent \textbf{Step 5.} To estimate $F_2\left( \langle\langle X, dY_\beta|_{p}-dY_\alpha|_{p}\rangle \rangle \right)$ for $\alpha,\beta\in N'_p$, we just need to estimate the following three items
\begin{align*}
(1)\ \ F_2&\left(    \langle  \langle X, dY_\alpha|_p\rangle\rangle -P_{\mathscr{F}_\alpha(p),x} \circ P_{p^2_\alpha,\mathscr{F}_\alpha(p)}\circ u^2_\alpha  \circ (u^1_\alpha)^{-1}\circ P_{p,p^1_\alpha}X          \right);\tag{6.12}\label{5.14}\\
(2)\ \ F_2&\left(    P_{\mathscr{F}_\alpha(p),x} \circ P_{p^2_\alpha,\mathscr{F}_\alpha(p)}\circ u^2_\alpha  \circ (u^1_\alpha)^{-1}\circ P_{p,p^1_\alpha}X -  \langle\langle X, dY_\alpha|_p\rangle \rangle     \right);\tag{6.13}\label{5.14'}\\
(3)\ \ F_2&\left(P_{\mathscr{F}_\alpha(p),x} \circ P_{p^2_\alpha,\mathscr{F}_\alpha(p)}\circ u^2_\alpha  \circ (u^1_\alpha)^{-1}\circ P_{p,p^1_\alpha}X \right.\\
&\left.\ -P_{\mathscr{F}_\beta(p),x} \circ P_{p^2_\beta,\mathscr{F}_\beta(p)}\circ u^2_\beta  \circ (u^1_\beta)^{-1}\circ P_{p,p^1_\beta}X         \right).\tag{6.14}\label{5.15}
\end{align*}
Here, $P_{p,q}$ denotes the parallel transformation along the normal minimal geodesic from $p$ to $q$.

We first estimate (\ref{5.14}) and (\ref{5.14'}).
 Given $\alpha\in N'_p$, set $s_1:=d(p^1_\alpha,p)$ and $s_2:=d(p^2_\alpha,\mathscr{F}_\alpha(p))$. Clearly,
there exists $Y\in T_{p^1_\alpha}M_1$ such that
\[
\left(\exp_{p^1_\alpha}\right)_{*\rho^1_\alpha(p)\cdot\nabla \rho^1_\alpha(p)}Y=X,
\]
where $\rho^1_\alpha(\cdot):=d(p^1_\alpha,\cdot)$. Now let
\begin{align*}
\overline{X}&:=u^2_\alpha\circ (u^1_\alpha)^{-1}\circ P_{p,p^1_\alpha}(X)\in T_{p^2_\alpha}M_2,\ \overline{Y}:=u^2_\alpha\circ (u^1_\alpha)^{-1}(Y)\in T_{p^2_\alpha}M_2,\\
l&:=d(x,\mathscr{F}_\alpha(p)),\ J_Y(s_1):=\left(\exp_{p^1_\alpha}\right)_{*s_1\nabla \rho^1_\alpha(p)}(s_1Y)\in T_{p}M_1,\\
J_{\overline{Y}}(s_2)&:=\left(\exp_{p^2_\alpha}\right)_{*s_2\nabla \rho^2_\alpha(\mathscr{F}_\alpha(p))}(s_2\overline{Y})\in T_{\mathscr{F}_\alpha(p)}M_2.
\end{align*}
Note that there exists $Z\in T_{x}M_2$ with
\[
\left(\exp_{x}\right)_{*ly}Z=\frac{1}{s_2}J_{\overline{Y}}(s_2)=\left(\exp_{p^2_\alpha}\right)_{*s_2\nabla \rho^2_\alpha(\mathscr{F}_\alpha(p))}(\overline{Y}),
\]
where $y:=\nabla \rho_{x}({\mathscr{F}_\alpha(p)})$ and $\rho_{x}(\cdot):=d(x,\cdot)$.
Thus, we have
\begin{align*}
Z&=\left(\exp_{x}\right)_{*ly}^{-1}\left(\exp_{p^2_\alpha}\right)_{*s_2\nabla \rho^2_\alpha(\mathscr{F}_\alpha(p))}(\overline{Y})\\
&=\left(\exp_{x}\right)_{*ly}^{-1}\left(\exp_{p^2_\alpha}\right)_{*s_2\nabla \rho^2_\alpha(\mathscr{F}_\alpha(p))}u^2_\alpha\circ (u^1_\alpha)^{-1}\left(\exp_{p^1_\alpha}\right)_{*\rho^1_\alpha(p)\cdot\nabla \rho^1_\alpha(p)}^{-1}X\\
&=\langle\langle X,dY_\alpha|_p\rangle\rangle.
\end{align*}
Since $F_i$ is Berwalden,
\begin{align*}
(\ref{5.14})\leq&F_2\left(P_{\mathscr{F}_\alpha(p),x}\left(\frac{1}{s_2}J_{\overline{Y}}(s_2)\right)- P_{\mathscr{F}_\alpha(p),x} \circ P_{p^2_\alpha,\mathscr{F}_\alpha(p)}\overline{X} \right)\\
&+ F_2\left(Z-P_{\mathscr{F}_\alpha(p),x}\left(\frac{1}{s_2}J_{\overline{Y}}(s_2)\right) \right)\\
\leq &F_2\left(\frac{1}{s_2}J_{\overline{Y}}(s_2)-  P_{p^2_\alpha,\mathscr{F}_\alpha(p)}\overline{Y} \right)+F_2\left(P_{p^2_\alpha,\mathscr{F}_\alpha(p)}\overline{Y}-  P_{p^2_\alpha,\mathscr{F}_\alpha(p)}\overline{X} \right)\\
&+F_2\left(Z-P_{\mathscr{F}_\alpha(p),x}\left(\frac{1}{s_2}J_{\overline{Y}}(s_2)\right) \right)\\
\leq &\sqrt{\Lambda}\left\|\frac{1}{s_2}J_{\overline{Y}}(s_2)-  P_{p^2_\alpha,\mathscr{F}_\alpha(p)}\overline{Y} \right\|_{T_1}+F_2\left(\overline{Y} -\overline{X} \right)\\
&+\sqrt{\Lambda}\left\| Z-P^{-1}_{x,\mathscr{F}_\alpha(p)}\left(\frac{1}{s_2}J_{\overline{Y}}(s_2)\right)\right\|_{T_2}\tag{6.15}\label{5.16}
\end{align*}
where $T_1$ is the velocity of the normal geodesic from $p^2_\alpha$ to $\mathscr{F}_\alpha(p)$, and $T_2$ is the  velocity of the normal geodesic from $x$ to $\mathscr{F}_\alpha(p)$.

Since $F_i$ is Berwalden, $P^{-1}_{p,p^1_\alpha}=P_{p^1_\alpha,p}$ and $P^{-1}_{x,\mathscr{F}_\alpha(p)}= P_{\mathscr{F}_\alpha(p),x}$. And it is easy to see that $s_1\leq\sqrt{\Lambda}R$, $s_2\leq\sqrt{\Lambda}R$ and $l<R$. Thus,
Lemma A.1 together with Lemma \ref{jestima1} and Corollary \ref{jestima2} yields
\begin{align*}
&\left\|\frac{1}{s_2}J_{\overline{Y}}(s_2)-  P_{p^2_\alpha,\mathscr{F}_\alpha(p)}\overline{Y} \right\|_{T_1}
\leq\frac{\Lambda^{2}R}{\mathfrak{s}_k(\sqrt{\Lambda}R)}\left( \frac{\mathfrak{s}_{-k}(\sqrt{\Lambda}R)}{\sqrt{\Lambda}R}-1   \right)F_1(X),\tag{6.16}\label{5.20}\\
&F_2\left(\overline{Y} -\overline{X} \right)
\leq\frac{\Lambda^{2}R}{\mathfrak{s}_k(\sqrt{\Lambda}R)}\left( \frac{\mathfrak{s}_{-k}(\sqrt{\Lambda}R)}{\sqrt{\Lambda}R}-1 \right)F_1(X),\tag{6.17}\label{5.24}\\
&\left\| Z-P^{-1}_{x,\mathscr{F}_\alpha(p)}\left(\frac{1}{s_2}J_{\overline{Y}}(s_2)\right)\right\|_{T_2}\leq
\frac{\Lambda^2 R}{\mathfrak{s}_{k}(R)}\left( \frac{\mathfrak{s}_{-k}(R)}{R}-1   \right)\frac{\mathfrak{s}_{-k}(\sqrt{\Lambda} R)}{\mathfrak{s}_{k}(\sqrt{\Lambda}R)} F_1(X).\tag{6.18}\label{5.19}
\end{align*}
By (\ref{5.16}), (\ref{5.20}), (\ref{5.24}) and (\ref{5.19}), we have
\[
(\ref{5.14})
\leq  \frac{3\Lambda^3 R}{\mathfrak{s}_k(\sqrt{\Lambda}R)}\left( \frac{\mathfrak{s}_{-k}(\sqrt{\Lambda}R)}{\sqrt{\Lambda}R}-1 \right) \frac{\mathfrak{s}_{-k}(\sqrt{\Lambda}R)}{\mathfrak{s}_k(\sqrt{\Lambda}R)}\cdot F_1(X).\tag{6.19}\label{5.25}
\]
Similarly, one can show
\[
(\ref{5.14'})
\leq  \frac{3\Lambda^3 R}{\mathfrak{s}_k(\sqrt{\Lambda}R)}\left( \frac{\mathfrak{s}_{-k}(\sqrt{\Lambda}R)}{\sqrt{\Lambda}R}-1 \right) \frac{\mathfrak{s}_{-k}(\sqrt{\Lambda}R)}{\mathfrak{s}_k(\sqrt{\Lambda}R)}\cdot F_1(X).\tag{6.20}\label{5.25'}
\]

We now estimate (\ref{5.15}). Given $\alpha,\beta\in N'_p$, set
\begin{align*}
&X_\alpha:=P_{p,p^1_\alpha}X\in T_{p^1_\alpha}M_1,\ X_\beta:=P_{p,p^1_\beta}X\in T_{p^1_\beta}M_1,\ X'_\alpha:=P_{p^1_\alpha,p^1_\beta}X_\alpha\in T_{p^1_\beta}M_1,\\
&\overline{X_\beta}:=u^2_\beta\circ (u^1_\beta)^{-1}(X_\beta)\in T_{p^2_\beta}M_2,\ \overline{X_\alpha}:=u^2_\alpha\circ (u^1_\alpha)^{-1}(X_\alpha)\in T_{p^2_\alpha}M_2,\\
&\overline{X'_\alpha}:=P_{p^2_\beta,p^2_\alpha}\circ u^2_\beta\circ (u^1_\beta)^{-1}(X'_\alpha)\in T_{p^2_\alpha}M_2,\ \overline{X'_\beta}:=P_{p^2_\beta,p^2_\alpha}\circ u^2_\beta\circ (u^1_\beta)^{-1}(X_\beta)\in T_{p^2_\alpha}M_2.
\end{align*}

Thus, we have
\begin{align*}
(\ref{5.15})\leq &F_2\left( P_{\mathscr{F}_\alpha(p),x} \circ P_{p^2_\alpha,\mathscr{F}_\alpha(p)}   \circ \overline{X_\alpha} -   P_{\mathscr{F}_\alpha(p),x} \circ P_{p^2_\alpha,\mathscr{F}_\alpha(p)}   \circ  \overline{X'_\alpha}\right)\\
&+F_2\left(
 P_{\mathscr{F}_\alpha(p),x} \circ P_{p^2_\alpha,\mathscr{F}_\alpha(p)}   \circ  \overline{X'_\alpha}-
 P_{\mathscr{F}_\beta(p),x} \circ P_{p^2_\beta,\mathscr{F}_\beta(p)} \circ\overline{X_\beta}     \right)\\
 \leq & F_2\left(   \overline{X_\alpha} -   \overline{X'_\alpha}  \right)+F_2\left(
 P_{\mathscr{F}_\alpha(p),x} \circ P_{p^2_\alpha,\mathscr{F}_\alpha(p)}   \overline{X'_\alpha}- P_{\mathscr{F}_\alpha(p),x} \circ P_{p^2_\alpha,\mathscr{F}_\alpha(p)} \overline{X'_\beta} \right)\\
 &+F_2\left(P_{\mathscr{F}_\alpha(p),x} \circ P_{p^2_\alpha,\mathscr{F}_\alpha(p)} \overline{X'_\beta} -
 P_{\mathscr{F}_\beta(p),x} \circ P_{p^2_\beta,\mathscr{F}_\beta(p)} \overline{X_\beta}   \right)\\
 \leq&  F_2\left(   \overline{X_\alpha} -   \overline{X'_\alpha}  \right)+F_2\left( \overline{X'_\alpha}- \overline{X'_\beta}\right)\\
 &+F_2\left(P_{\mathscr{F}_\alpha(p),x} \circ P_{p^2_\alpha,\mathscr{F}_\alpha(p)} \overline{X'_\beta} -
 P_{\mathscr{F}_\beta(p),x} \circ P_{p^2_\beta,\mathscr{F}_\beta(p)} \overline{X_\beta}   \right)\tag{6.21}\label{5.26}
\end{align*}

Firstly, we have
\begin{align*}
&F(\overline{X_\alpha}-\overline{X_\alpha'})=F\left(P^{-1}_{p^2_\beta,p^2_\alpha}\circ u^2_\alpha\circ (u^1_\alpha)^{-1} (X_\alpha)-u^2_\beta\circ (u^1_\beta)^{-1}   (X'_\alpha)\right)\\
\leq &\sqrt{\Lambda}\left\|(u^2_\beta)^{-1}\circ P_{p^2_\alpha,p^2_\beta} \circ u^2_\alpha\circ (u^1_\alpha)^{-1} (X_\alpha)- (u^1_\beta)^{-1}   (X'_\alpha)   \right\|\\
= & \sqrt{\Lambda}\left\|\mathfrak{g}^2_{\beta\alpha}((u^1_\alpha)^{-1}(X_\alpha))-\mathfrak{g}^1_{\beta\alpha}((u^1_\alpha)^{-1}(X_\alpha))\right\|\leq \Lambda\cdot \varepsilon_2\cdot F_1(X).\tag{6.22}\label{5.27}
\end{align*}
Secondly, Lemma \ref{parallelesi} yields
\begin{align*}
F_2(\overline{X'_\alpha}- \overline{X'_\beta})\leq \Lambda\cdot F_1(X'_\alpha-X_\beta)
\leq   \mathfrak{C}(n,k,\Lambda)\cdot\Lambda^3\cdot F_1(X)\cdot R^2,\tag{6.23}\label{5.28}
\end{align*}
where $\mathfrak{C}(n,k,\Lambda)$ is the constant as in Lemma \ref{parallelesi}. Since $\alpha,\beta\in N'_p$, $\phi^2_\alpha(\overline{\mathcal {B}_0(R)})\cap \phi^2_\beta(\overline{\mathcal {B}_0(R)})\neq\emptyset$, which implies that $d(p^2_\alpha,p^2_\beta)<2\Lambda R$.
By Lemma \ref{parallelesi} again, we have
\begin{align*}
&F_2\left(P_{\mathscr{F}_\alpha(p),x} \circ P_{p^2_\alpha,\mathscr{F}_\alpha(p)} \overline{X'_\beta} -
 P_{\mathscr{F}_\beta(p),x} \circ P_{p^2_\beta,\mathscr{F}_\beta(p)} \overline{X_\beta}   \right)\\
 \leq & F_2\left(P_{\mathscr{F}_\alpha(p),x} \circ P_{p^2_\alpha,\mathscr{F}_\alpha(p)} \circ P_{p^2_\beta,p^2_\alpha} \overline{X_\beta}-
  P_{p^2_\beta,x} \overline{X_\beta}   \right)\\
  &+F_2\left(P_{p^2_\beta,x} \overline{X_\beta}  -
 P_{\mathscr{F}_\beta(p),x} \circ P_{p^2_\beta,\mathscr{F}_\beta(p)} \overline{X_\beta}     \right)\\
 \leq &F_2\left(P_{\mathscr{F}_\alpha(p),x} \circ P_{p^2_\alpha,\mathscr{F}_\alpha(p)} \circ P_{p^2_\beta,p^2_\alpha} \overline{X_\beta}-
  P_{p^2_\beta,x} \overline{X_\beta}   \right)+4\mathfrak{C}(n,k,\Lambda)\cdot\Lambda^{\frac52}\cdot F_1(X)\cdot R^2\\
  \leq &F_2\left(P_{\mathscr{F}_\alpha(p),x} \circ P_{p^2_\alpha,\mathscr{F}_\alpha(p)} \circ P_{p^2_\beta,p^2_\alpha} \overline{X_\beta}-
 P_{\mathscr{F}_\alpha(p),x} \circ P_{p^2_\beta,\mathscr{F}_\alpha(p)} \overline{X_\beta} \right)\\
  &+F_2\left(P_{\mathscr{F}_\alpha(p),x} \circ P_{p^2_\beta,\mathscr{F}_\alpha(p)} \overline{X_\beta}-
  P_{p^2_\beta,x} \overline{X_\beta}   \right)+4\mathfrak{C}(n,k,\Lambda)\cdot\Lambda^{\frac52}\cdot F_1(X)\cdot R^2\\
  \leq &29\Lambda^3\cdot \mathfrak{C}(n,k,\Lambda)\cdot R^2\cdot F_1(X).\tag{6.24}\label{5.29}
\end{align*}
Now by (\ref{5.26}),  (\ref{5.27}),  (\ref{5.28}) and  (\ref{5.29}), we obtain
\begin{align*}
(\ref{5.15})\leq  \left[ 30 \Lambda^3 \mathfrak{C}(n,k,\Lambda) R^2+\Lambda \varepsilon_2\right]   F_1(X).\tag{6.25}\label{5.30}
\end{align*}

The triangle inequality then yields
\begin{align*}
& F_2(\langle\langle X,dY_\beta|_p-dY_\alpha|_p   \rangle \rangle) \leq (\ref{5.14})_\beta+(\ref{5.15})_{\beta\alpha}+(\ref{5.14'})_\alpha\\
\leq& (\ref{5.25})_\beta+(\ref{5.30})_{\beta\alpha}+(\ref{5.25'})_{\alpha}=\mathcal {C}_3(n,k,\Lambda,R,\varepsilon_2)F_1(X),
\end{align*}
which together with (\ref{5.12}) and
(\ref{5.13}) yields (\ref{positv}) and
\begin{align*}
F_2(\text{II})\geq  \left[ \frac{(1-kR^2)}{\Lambda^5}-\mathcal {C}_3(n,k,\Lambda,R,\varepsilon_2)\right]\cdot F_1(X).\tag{6.26}\label{5.32}
\end{align*}

Step 3 furnishes that
\begin{align*}
F_2(-\,\text{I})\leq \sqrt{\Lambda}\cdot F_2(\text{I})\leq \frac{2^{2n+6}\Lambda^{4n+5+\frac12}}{\mathfrak{s}_k(\sqrt{\Lambda}R)}\varepsilon_1 F_1(X).\tag{6.27}\label{5.33}
\end{align*}
Thus, (\ref{5.6}) together with (\ref{5.32}) and (\ref{5.33}) yields
\begin{align*}
&F_2\left(\left.\frac{d}{dt}\right|_{t=0}\mathcal {V}(\gamma(t),x) \right)=F_2(\text{I}+\text{II})\geq F_2(\text{II})-F_2(-\text{I})\\
\geq &\left[ \frac{(1-kR^2)}{\Lambda^5}-\mathcal {C}_{3}(n,k,\Lambda,R,\varepsilon_2)-\frac{2^{2n+6}\Lambda^{4n+6}}{\mathfrak{s}_k(\sqrt{\Lambda}R)}\varepsilon_1 \right]\cdot F_1(X)>0,
\end{align*}
which implies that $\mathscr{F}_*$ is nonsingular (See Step 2).

\noindent{\textbf{Step 6.}} Since $\mathscr{F}$ is a local diffeomorphism, we can define a new Finsler metric $\widetilde{F}_1$ on $M_1$ by $\widetilde{F}_1:= \mathscr{F}^*F_2$. Thus, $(M_1,\widetilde{F}_1)$ is a forward geodesically complete Finsler manifold, since $M_1$ is closed. It follows from \cite[Theorem 9.2.1]{BCS} that $\mathscr{F}:M_1\rightarrow M_2$ is a covering projection.

Let $\mathscr{G}:M_2\rightarrow M_1$ be the map constructed as $\mathscr{F}$. Given any point $p\in M_1$, there exists a point $p^1_\alpha\in M_1$ such that $d(p^1_\alpha,p)<R/(2\Lambda^\frac32)$, which implies
\begin{align*}
d(\mathscr{F}_\alpha(p^1_\alpha),\mathscr{F}_\alpha(p))=F_2(u^2_\alpha\circ (\phi^1_\alpha)^{-1}(p))\leq \Lambda d(p^1_\alpha,p)<R/(2\sqrt{\Lambda}).\tag{6.28}\label{6.28}
\end{align*}
Since $\alpha\in N'_p$, $d(\mathscr{F}_\alpha(p),\mathscr{F}(p))<R/(2\sqrt{\Lambda})$, which together with (\ref{6.28}) yields
\[
d(p^2_\alpha,\mathscr{F}(p))=d(\mathscr{F}_\alpha(p^1_\alpha),\mathscr{F}(p))<R/\sqrt{\Lambda},
\]
that is, $\mathscr{F}(p)\in \phi^2_\alpha(\overline{\mathcal {B}_0(R)})$. Set $\mathscr{G}_\alpha:=\phi^1_\alpha\circ(\phi^2_\alpha)^{-1}:\phi^2_\alpha(\overline{\mathcal {B}_0(R)})\rightarrow \phi^1_\alpha(\overline{\mathcal {B}_0(R)})$. The same argument as before yields that
\begin{align*}
d(p^1_\alpha,\mathscr{G}\circ \mathscr{F}(p))\leq d(\mathscr{G}_\alpha(p^2_\alpha),\mathscr{G}_\alpha(\mathscr{F}(p)))+d(\mathscr{G}_\alpha(\mathscr{F}(p)),\mathscr{G}(\mathscr{F}(p)))<2\sqrt{\Lambda}R,
\end{align*}
and therefore, $\mathscr{G}\circ \mathscr{F}(p)\in B^+_p(3\sqrt{\Lambda}R)$. Likewise, one can show $\mathscr{F}\circ\mathscr{G}(q)\in B^+_q(3\sqrt{\Lambda}R)$. That is, both $\mathscr{G}\circ \mathscr{F}$ and $\mathscr{F}\circ\mathscr{G}$ map every point
to a convex neighborhood of itself and hence, they are homotopic to the
identity. Now we conclude that $\mathscr{F}$ and $\mathscr{G}$ are diffeomorphisms.
\end{proof}

\appendix
\section{Some estimates for Jacobi fields}
In this section, we always assume that $(M,F)$ be a compact Finsler $n$-manifold with
$\Lambda_F\leq \Lambda$ and $|\mathbf{K}_M|\leq k$. Given $y\in SM$, we use $\gamma_y(t)$ to denote the normal geodesic with $\dot{\gamma_y}(0)=y$.

\begin{lemma}\label{Jacbi}For any $y\in S_pM$ and $X\in T_pM-\{0\}$, we have
\[
\frac{\mathfrak{s}_k(t)}{t}\leq \frac{\|(\exp_p)_{*ty}X\|_T}{\|X\|_T}\leq \frac{\mathfrak{s}_{-k}(t)}{t},\ \ t\in \left[0,\frac{\pi}{2\sqrt{k}}\right],
\]
where $\|\cdot\|_T:=g_{T}(\cdot,\cdot)$ and $T:=\dot{\gamma}_y(t)$.
\end{lemma}
\begin{proof}Let
\[
J^\bot(t):=(\exp_p)_{*ty}\left(tX^\bot\right),\ J^\parallel(t):=(\exp_p)_{*ty}\left(\alpha ty\right),
\]
where $\alpha:=g_y(y,X)$ and $X^\bot:=X-\alpha y$. Since $|\mathbf{K}_M|\leq k$, it follows from the Rauch theorem \cite[Theorem 9.6.1]{BCS} that
\[
\mathfrak{s}_{k}(t)\|X^\bot\|_T\leq \|J^\bot(t)\|_T\leq \mathfrak{s}_{-k}(t)\|X^\bot\|_T, \ \ t\in \left[0,\frac{\pi}{2\sqrt{k}}\right].
\]
Note that $\|J^\parallel(t)\|_T=t|\alpha|$. Thus,
\[
\mathfrak{s}_{k}(t)\|\alpha y\|_T\leq \|J^\parallel(t)\|_T\leq \mathfrak{s}_{-k}(t)\|\alpha y\|_T, \ \ t\in \left[0,\frac{\pi}{2\sqrt{k}}\right].
\]
The lemma follows from the inequalities above.
\end{proof}

\begin{lemma}\label{firstjab}
Given three points $p,q,x\in M$. Let $\gamma(s)$, $s\in [0,1]$ be a smooth curve from $p$ to $q$ such that $d(x,\gamma(s))<\min\{\mathfrak{i}_M,\frac{\pi}{2\sqrt{k}}\}$ for all $s$.
Set $P:=\exp^{-1}_x(p)$ and $Q:=\exp^{-1}_x(q)$.

\bigskip
(1) Suppose that $\gamma(s)$ is a minimal geodesic from $p$ to $q$. Then
\[
\frac{1}{\Lambda}\min_{s\in [0,1]}\frac{\mathfrak{s}_{k}(d(x,\gamma(s))}{ d(x,\gamma(s))}F(Q-P)\leq d(p,q).
\]

\bigskip
(2) Suppose that $\exp_x^{-1}(\gamma(s))$ is a straight line from $P$ to $Q$. Then
\[
d(p,q)\leq \Lambda\max_{s\in [0,1]}\frac{\mathfrak{s}_{-k}(d(x,\gamma(s))}{ d(x,\gamma(s))}F(Q-P).
\]

\bigskip
(3) Suppose that $\gamma([0,1])\subset B^+_x(R)$, where $R<\min\{\mathfrak{i}_M,\frac{\pi}{2\sqrt{k}}\}$. Then
\[
\frac{\mathfrak{s}_{k}(R)}{\Lambda\cdot R}F(Q-P)\leq d(p,q)\leq \frac{\Lambda\cdot\mathfrak{s}_{-k}(R)}{ R}F(Q-P).
\]
\end{lemma}
\begin{proof}
For each $s\in [0,1]$, there exists $V_s\in T_xM$ such that $\exp_xV_s=\gamma(s)$. We define a geodesic variation
\[
\sigma(t,s):=\exp_p(tV_s), \ (t,s)\in [0,1]\times[0,1].
\]
Set
\[
T:=\frac{\partial\sigma}{\partial t}=(\exp_x)_{*tV_s}V_s,\ U:=\frac{\partial\sigma}{\partial s}=(\exp_x)_{*tV_s}(t\dot{V}_s),
\]
where $\dot{V}_s:=\frac{dV_s}{ds}$.
It follows from Lemma \ref{Jacbi} that
\[
\frac{1}{\Lambda}\min_{s\in [0,1]}\frac{\mathfrak{s}_{k}(d(x,\gamma(s))}{ d(x,\gamma(s))}\int^1_0F(\dot{V}_s)ds\leq \int^1_0F(U(1,s))ds\leq \Lambda\max_{s\in [0,1]}\frac{\mathfrak{s}_{-k}(d(x,\gamma(s))}{ d(x,\gamma(s))}\int^1_0F(\dot{V}_s)ds.
\]

(1) Suppose that $\gamma(s)$ is a minimal geodesic from $p$ to $q$. Note that $U(1,s)=\dot{\gamma}(s)$. Hence,
\[
F(Q-P)\leq \int^1_0F(\dot{V}_s)ds,\ \int^1_0F(U(1,s))ds=d(p,q).
\]

(2) Suppose that $\exp_x^{-1}(\gamma(s))$ is a straight line from $P$ to $Q$. Thus,
\[
F(Q-P)= \int^1_0F(\dot{V}_s)ds,\ \int^1_0F(U(1,s))ds\geq d(p,q).
\]

Clearly, (3) follows from (1) and (2).
\end{proof}

Recall the definition of curvature operator $\mathcal {R}$ of a Finsler manifold (cf. \cite{ZY}): Given $p\in M$ and $y\in S_pM$.
Let $P_{t;y}$ denote the parallel transformation along the geodesic $\gamma_y(t)$ from $T_pM$ to $T_{\gamma_y(t)}M$. The curvature operator $\mathcal {R}$ is defined by
\[
\mathcal {R}(t;y):=P^{-1}_{t;y}\circ R_T\circ P_{t;y}: \, y^\bot\rightarrow y^\bot,
\]
where $R_T:=R_T(\cdot,T)T$ and  $y^\perp:=\{W\in T_pM: \ g_y(y,W)=0\}$.

\begin{lemma}\label{normcurvature}Set
\[
\|\mathcal {R}(t;y)\|:=\sup_{X\in y^\bot-\{0\}}\frac{\|\mathcal {R}(t;y)X\|_y}{\|X\|_y},
\]
where $\|\cdot\|_y:=\sqrt{g_y(\cdot,\cdot)}$. Thus, $\|\mathcal {R}(t;y)\|\leq k$.
\end{lemma}
\begin{proof}
Let $\{\xi_\alpha\}$ and $\{e_\alpha\}$ denote the eigenvalues and eigenvectors of $\mathcal {R}(t;y)$, respectively. Since $\mathcal {R}$ is self-adjoint, $\{e_\alpha\}$ is an orthonormal basis for $y^\bot$.
Then
\[
\mathcal {R}(t;y)e_\alpha=\xi_\alpha e_\alpha\Rightarrow \langle \mathcal {R}(t;y)e_\alpha, e_\alpha\rangle =\xi_\alpha,
\]
where $\langle\cdot,\cdot\rangle:=g_y(\cdot,\cdot)$.
Note that
\[
K(T,P_{t;y}e_\alpha)=g_T(R_T(P_{t;y}e_\alpha),P_{t;y}e_\alpha)=\langle\mathcal {R}(t;y)e_\alpha,e_\alpha\rangle,
\]
where $T=\dot{\gamma_y}(t)$. Hence, $-k\leq \xi_\alpha \leq k$, which implies that $\|\mathcal {R}(t;y)\|\leq k$.
\end{proof}

Using  Lemma \ref{normcurvature} and the same argument as in \cite[Theorem IX. 4.1, Corollary IX. 4.3]{C2}, one can show that
\begin{lemma}\label{Chavel2}Consider the vector equation of $\eta(t)\in y^\bot$:
\[
\eta''+\mathcal {R}(t,y)\eta=0.
\]
If $\eta(0)=0$, then
\[
\|\eta(s)-s\eta'(0)\|_y\leq \|\eta'(0)\|_y\cdot\left( \mathfrak{s}_{-k}(s)-s    \right)
\]
for all $s>0$, where $\|\cdot\|_y:=\sqrt{g_y(\cdot,\cdot)}$.
\end{lemma}

In particular, let $\mathcal {A}(t,y)$ be the solution of the matrix (or linear
transformation) ordinary differential equation on $y^\bot$:
\[
\left \{
\begin{array}{lll}
&\mathcal {A}{''}+\mathcal {R}(t;y)\mathcal {A}=0,\\
&\mathcal {A}(0;y)=0,\\
 &\mathcal {A}'(0;y)=\mathcal {I}.
\end{array}
\right.
\]
Then $P_{t;y}\mathcal {A}(t, y)X = (\exp_{p})_{*ty}tX$, for any $X\in y^\bot$. Now we have the following
\begin{lemma}\label{jestima1}Given $y\in S_pM$ and $X\in y^\bot$, we have
\[
\left\|  (\exp_p)_{*ty}X-P_{t;y}X             \right\|_T\leq \left( \frac{\mathfrak{s}_{-k}(t)}{t}-1   \right)\|X\|_T,
\]
where $T:=\dot{\gamma}_y(t)$ and $\|\cdot\|_T:=\sqrt{g_T(\cdot,\cdot)}$.
\end{lemma}
\begin{proof}
 Set $\eta:=\mathcal {A}(t;y)X$. Clearly, $\eta(0)=0$ and $\eta'(0)=X$. By Lemma \ref{Chavel2}, we have
\[
\|\mathcal {A}(t;y)X-tX\|_T\leq\left( \mathfrak{s}_{-k}(t)-t    \right)\|X\|_T.
\]
It should be noted that $\|W\|_T=\|P_{t;y}W\|_T$ for any $W\in T_pM$. Hence,
\[
\left\|\frac{P_{t;y}\mathcal {A}(t;y)X}{t}-P_{t;y}X\right\|_T\leq \left( \frac{\mathfrak{s}_{-k}(t)}{t}-1    \right)\|X\|_T.
\]
\end{proof}
\begin{remark}
If $X=ky$ for any $k\in \mathbb{R}$, then
\[
 (\exp_p)_{*ty}X=P_{t;y}X.
\]
Hence, Lemma \ref{jestima1} holds for all $X\in T_pM$.
\end{remark}
\begin{corollary}
\label{jestima2}Given $y\in S_pM$ and $Y\in T_{\dot{\gamma}_y(t)}M$, where $0\leq t< \frac{\pi}{2\sqrt{k}}$. Then
\[
\left\|  (\exp_p)^{-1}_{*ty}Y-P_{t;y}^{-1}Y             \right\|_T\leq \frac{t}{\mathfrak{s}_{k}(t)}\left( \frac{\mathfrak{s}_{-k}(t)}{t}-1   \right)\|Y\|_T,
\]
where $T:=\dot{\gamma}_y(t)$ and $\|\cdot\|_T:=\sqrt{g_{T}(\cdot,\cdot)}$.
\end{corollary}
\begin{proof}Since $0\leq t< \frac{\pi}{2\sqrt{k}}$, there exists a unique  $X\in T_pM$ such that
\[
Y=(\exp_p)_{*ty}X.
\]
Then Lemma \ref{Jacbi} together with Lemma \ref{jestima1} yields that
\begin{align*}
\frac{\mathfrak{s}_k(t)}{t}{\|X-P^{-1}_{t;y}Y\|_T}&\leq {\|(\exp_p)_{*ty}(X-P^{-1}_{t;y}Y)\|_T}=\|Y-(\exp_p)_{*ty}P^{-1}_{t;y}Y\|_T\\
&= \|P_{t;y}P^{-1}_{t;y}Y-(\exp_p)_{*ty}P^{-1}_{t;y}Y\|_T\\
&\leq \left( \frac{\mathfrak{s}_{-k}(t)}{t}-1   \right)\|P^{-1}_{t;y}Y\|_T=\left( \frac{\mathfrak{s}_{-k}(t)}{t}-1   \right)\|Y\|_T.
\end{align*}
\end{proof}

\begin{lemma}\label{newJacobies}
Let $\gamma(t)$, $t\geq 0$ be a unit speed speed geodesic. Then there exists two positive constants $\mathfrak{t}=\mathfrak{t}(n,k,\Lambda)$ such that for any Jacobi field $J(t)$ along $\gamma$ with $J(0)=0$, we have
\[
\|J(t)-tJ'(t)\|_T\leq \frac{1}{20\Lambda}\|J(t)\|_T,\ \ t\in [0,\mathfrak{t}],
\]
where $T:=\dot{\gamma}(t)$ and $\|\cdot\|_T:=\sqrt{g_{T}(\cdot,\cdot)}$.
\end{lemma}
\begin{proof}
Clearly, we have
\begin{align*}
\frac{d}{dt}g_T(J(t)-tJ'(t),J(t)-tJ'(t))\leq2\|tJ''(t)\|_T\cdot\|J(t)-tJ'(t)\|_T,
\end{align*}
which implies that
\[
\frac{d}{dt}\|J(t)-tJ'(t) \|_T\leq \|tJ''(t)\|_T=\|tR_T(J,T)T\|_T.
\]
Lemma \ref{normcurvature} implies that
\[
\|R_T(J,T)T\|_T=\|R_T\circ P_{t;y}\mathcal {A}J(0)\|_T=\|\mathcal {R}(t;y)\mathcal {A}J(0)\|_T\leq k\cdot\|J(t)\|_T.
\]
From above, we obtain
\[
\frac{d}{dt}\|J(t)-tJ'(t) \|_T\leq kt\|J(t)\|_T.\tag{A.1}\label{*4}
\]
The Rauch comparison theorem yields
\[
\|J(t)\|_T\leq \|J'(0)\|_T\cdot \mathfrak{s}_{-k}(t),
\]
for $t\in [0,\frac{\pi}{2\sqrt{k}}]$.
(\ref{*4}) then furnishes
\[
\frac{d}{dt}\|J(t)-tJ'(t) \|_T\leq k\|J'(0)\|_T\cdot t\mathfrak{s}_{-k}(t),
\]
which implies that
\begin{align*}
\|J(t)-tJ'(t) \|_T&\leq \frac{1}{\sqrt{k}}\cdot\|J'(0)\|_T\cdot\left[ \sqrt{k}\,t\cosh\sqrt{k}\,t -\sinh\sqrt{k}\,t  \right] \\
&\leq  \frac{\sqrt{k}\,t\cosh\sqrt{k}\,t -\sinh\sqrt{k}\,t }{\sqrt{k}\cdot\mathfrak{s}_{k}(t)}\cdot \|J(t)\|_T.
\end{align*}
Since
\[
\lim_{t\rightarrow 0^+}\frac{\sqrt{k}\,t\cosh\sqrt{k}\,t -\sinh\sqrt{k}\,t }{\sqrt{k}\cdot\mathfrak{s}_{k}(t)}=0,
\]
there exists some $\mathfrak{t}=\mathfrak{t}(n,k,\Lambda)\in (0,\frac{\pi}{2\sqrt{k}})$ such that for $t\in [0,\mathfrak{t}]$.
\[
0<\frac{\sqrt{k}\,t\cosh\sqrt{k}\,t -\sinh\sqrt{k}\,t }{\sqrt{k}\cdot\mathfrak{s}_{k}(t)}<\frac{1}{20\Lambda}.
\]
\end{proof}

\section{Some estimates on Berwald manifolds}
In this section, we always assume that  $(M,F)$ is a Berwald manifold with $|\mathbf{K}_M|\leq k$ and $\Lambda_F\leq \Lambda$.

\begin{lemma}\label{CURVA}
Given $X$, $Y$, $W$ and $T\in S_pM$, we have
\[
|R_T(X,Y,T,W)|\leq \frac{2}{3}\Lambda^\frac32 k (1+\sqrt{\Lambda})^2.
\]
\end{lemma}
\begin{proof} Lemma \ref{normcurvature} yields that
\[
|R_U(X,Y,T,X)|=|g_U(R(T,X)X,Y)|\leq \|R(T,X)X\|_U\cdot\|Y\|_U \leq \Lambda^\frac32\cdot k,\tag{B.1}\label{B.1}
\]
where $\|\cdot\|_U:=\sqrt{g_U(\cdot,\cdot)}$.
A direct calculation shows that
\begin{align*}
6R_T(X,Y,T,W)
=&-R_T(W+X,Y,W+X,T)+R_T(W-X,Y,W-X,T)\\
&-R_T(T-X,Y,T-X,W)+R_T(T+X,Y,T+X,W).
\end{align*}
Then (\ref{B.1}) furnishes that
\begin{align*}
|R_T(W+X,Y,W+X,T)|&=\left|R_T\left(\frac{W+X}{F(W+X)},Y,\frac{W+X}{F(W+X)},T\right)\right|F^2(W+X)\\
&\leq \Lambda^\frac32 k [F(W)+F(X)]^2=4\Lambda^\frac32 k.
\end{align*}
\begin{align*}
|R_T(W-X,Y,W-X,T)|&=\left|R_T\left(\frac{W-X}{F(W-X)},Y,\frac{W-X}{F(W-X)},T\right)\right|F^2(W-X)\\
&\leq \Lambda^\frac32 k [F(W)+F(-X)]^2=\Lambda^\frac32 k (1+\sqrt{\Lambda})^2.
\end{align*}
Hence, we obtain
\[
|R_T(X,Y,T,W)|\leq \frac{2}{3}\Lambda^\frac32 k (1+\sqrt{\Lambda})^2.
\]
\end{proof}

\begin{lemma}\label{triangle}Let $Y(t)$ be a smooth vector filed along a constant speed geodesic $\gamma(t)$. Then
\[
\frac{d}{dt}\|Y(t)\|\leq \|\nabla_T Y\|,
\]
where $\nabla$ is the Chern connection, $T:=\dot{\gamma}(t)$ and $\|\cdot\|$ is the norm induced by the average Riemannian metric $\tilde{g}$.
\end{lemma}
\begin{proof}
Denote by $P_t$ the parallel transportation along $\gamma$ from $T_{\gamma(0)}M$ to $T_{\gamma(t)}M$. Choose a basis $\{e_i\}$ for $T_{\gamma(0)}M$. Then $E_i(t):=P_t e_i$, $1\leq i \leq n$, is a basis of $T_{\gamma(t)}M$.
For any $w\in T_{\gamma(0)}M-\{0\}$, we have
\begin{align*}
\frac{d}{dt} g_{(\gamma(t), P_tw)}(E_i(t),E_j(t))=\frac{2}{F(P_tw)}A_{(\gamma(t), P_tw)}\left(E_i(t),E_j(t),\nabla_{\dot{\gamma}} P_t w\right)=0.\tag{B.2}\label{B.2'}
\end{align*}
Since $(M,F)$ is a Berwald manifold,
\[
P_t(B_{\gamma(0)}M)=B_{\gamma(t)}M,\ \vol(x)=\text{const},
\]
 where $B_xM:=\{y\in T_xM: F(x,y)<1\}$ and $\vol(x)$ is the Riemannian volume of $S_xM$ (see \cite[Lemma 5.3.2]{Sh1} and \cite{BC1}).
Denote by $(y^i)$ (resp. $(z^i)$) the corresponding coordinate system in $T_{\gamma(0)}M$ (resp. $T_{\gamma(t)}M$) with respect to $\{e_i\}$ (resp. $\{E_i\}$). Thus, $z^i\circ P_t=y^i$. Now (\ref{B.2'}) together with Stokes' formula yields
\begin{align*}
&\tilde{g}_{\gamma(t)}(E_i(t),E_j(t))=\frac{n}{\vol(\gamma(t))}\int_{v\in B_{\gamma(t)}M} g_{(\gamma(t),v)}(E_i(t),E_j(t))dz^1\wedge\cdots \wedge dz^n\\
=&\frac{n}{\vol(\gamma(t))}\int_{w\in B_{\gamma(0)}M} g_{(\gamma(t),P_t w)}(P_t e_i,P_t e_j)P_t^*dz^1\wedge\cdots \wedge P_t^*dz^n\\
=&\frac{n}{\vol(\gamma(0))}\int_{w\in B_{\gamma(0)}M} g_{(\gamma(0),w)}(e_i,e_j)dy^1\wedge\cdots \wedge dy^n=\tilde{g}_{\gamma(0)}(e_i,e_j),
\end{align*}
which implies that
\begin{align*}
2\|Y\|\frac{d}{dt}\|Y\|=\frac{d}{dt}\tilde{g}_{\gamma(t)}(Y(t),Y(t))=2\tilde{g}_{\gamma(t)}(\nabla_TY,Y)\leq 2\|\nabla_TY\|\cdot\|Y\|.
\end{align*}
\end{proof}

\begin{remark}For the Busemann-Hausdorff measure, the
S-curvature of a Berwald manifold always vanishes (see \cite{Sh1}).
The same argument as above implies that
for the
Holmes-Thompson measure, the S-curvature of a Berwald manifold also vanishes.
\end{remark}

\begin{lemma}\label{parallelesi}
Given three points $p_1$, $p_2$ and $p_3$  in $M$, let  $\sigma_{ij}(t)$, $0\leq t \leq  1$ denote the minimizing constant speed geodesic from $p_i$ to $p_j$.
We construct a geodesic variation $\sigma(s,t):[0,1]\times [0,1]\rightarrow M$:

\bigskip

(1) $\sigma(s,0)=p_1$ and $\sigma(s,1)=p_3$;

(2) Let $p_4$ be the mid point in $\sigma_{13}$, that is, $d(p_1,p_4)=d(p_4,p_3)$. Let $\sigma_{24}(s)$, $s\in [0,1]$ be the minimal geodesic from $p_2$ to $p_4$. For each $s\in [0,1]$, $\sigma_s(t)$, $t\in [0,\frac12]$ be a constant geodesic from $p_1$ to $\sigma_{24}(s)$, and $\sigma_s(t)$, $t\in [\frac12,1]$ be a constant geodesic from $\sigma_{24}(s)$ to $p_3$. Hence, $\sigma_s(t)$, $t\in [0,1]$ be a piecewise geodesic from $p_1$ to $p_3$.

\bigskip

Suppose that $\triangle_{p_1p_2p_3}\subset B^+_{p_1}(R)$, where $R<\min\{\mathfrak{i}_M,\frac{\pi}{8\sqrt{k\Lambda}}\}$.
Given  a vector in $X\in T_{p_1}M$, Set $X_{13}:=P_{\sigma_{13}}X$ and $X_{123}:=P_{\sigma_{23}}P_{\sigma_{12}}X$, where $P_{\sigma_{ij}}$ is the parallel translation along $\sigma_{ij}$. Then there exits a positive number $\mathfrak{C}(n,k,\Lambda)$ such that
\[
F(X_{123}-X_{13})\leq \mathfrak{C}(n,k,\Lambda)\cdot F(X)\cdot R^2.
\]
\end{lemma}
\begin{proof}
\textbf{Step 1.}
Set $T:=\sigma_*\frac{\partial}{\partial t}$, $U:=\sigma_*\frac{\partial}{\partial s}$. It should be noted that $U$ is a Jacobi field. Since $\triangle_{p_1p_2p_3}\subset B^+_{p_1}(R)$, we have $F(T)\leq 4 R\sqrt{\Lambda}<\frac{\pi}{2\sqrt{k}}$.
Clearly,
\begin{align*}
U\left(s,\frac12\right)=\frac{d}{ds}\sigma\left(s,\frac12\right)=\frac{d}{ds}\sigma_{24}(s), \ d(p_2,p_4)=\int^1_0F\left(\frac{d}{ds}\sigma_{24}(s)\right)ds=F\left(\frac{d}{ds}\sigma_{24}(s)\right).
\end{align*}
Hence,
\[
F\left(U\left(s,\frac12\right)\right)<2R\sqrt{\Lambda}.\tag{B.3}\label{B.2}
\]
Note that for each fixed $s\in [0,1]$, there exists $Y_s\in T_{p_1}M$ such that
\[
U(s,t)=\left(\exp_{p_1}\right)_{*2tT(s,0)}tY_s, \ t\in \left[0,\frac12\right].
\]
It follows from Lemma \ref{Jacbi} that
\[
F(tY_s)\leq \frac{t\Lambda F(T)}{\mathfrak{s}_k(tF(T))}F(U(s,t)),
\]
which together with (\ref{B.2}) then yields that
\[
F(Y_s)\leq 2 \Lambda \frac{\frac12F(T)}{\mathfrak{s}_k(\frac12F(T))}F\left(U\left(s,\frac12\right)\right)<\frac{2\pi\Lambda^{\frac32}}{\sqrt{k}}R.
\]
Using Lemma \ref{Jacbi} again, we obtain that for $s\in [0,1]$ and $t\in [0,\frac12]$,
\[
F(U(s,t))\leq \Lambda\frac{\mathfrak{s}_{-k}(tF(T))}{tF(T)} F(tY_s)\leq 2{\mathfrak{s}_{-k}\left(\frac{\pi}{2\sqrt{k}}\right)}\Lambda^{\frac52}R.\tag{B.4}\label{*1}
\]
By consider the revised metric $\tilde{F}(y):=F(-y)$, the same argument then yields that (\ref{*1}) holds for $s\in [0,1]$ and $t\in [\frac12,1]$.

\noindent\textbf{Step 2.}
Let $X_t(s)=:P_{\sigma_s(t)}X$ denote the vector field on $\sigma([0,1]\times [0,1])$ induced by the parallel transformation along $\sigma_s(t)$. Thus, for any fixed $t\in [0,1]$, we have
\[
\nabla_sX_t:=\nabla_{U}X_t=\left[\frac{dX_t^i}{ds}+X^j_t\Gamma^i_{jk}U^k \right]\frac{\partial}{\partial x^i}.
\]
Since $X_0(s)=X$ and $X_1(s)\in T_{p_3}M$, it is easy to see that
\[
\lim_{t\rightarrow 0^+}\nabla_sX_t=0,\ \lim_{t\rightarrow 1^-}\nabla_sX_t=\frac{dX_1}{ds}(s).\tag{B.5}\label{B.6}
\]
Lemma \ref{triangle} together with (\ref{B.6}) implies
\begin{align*}
\|\nabla_sX_t\|_{t=1-\epsilon}=&\int^{1-\epsilon}_0\frac{d}{dt}\|\nabla_sX_t\|dt\leq \int^{1-\epsilon}_0\|\nabla_T\nabla_UX_t\|dt\\
=&\int^{1-\epsilon}_0\|R(T,U)X_t\|dt\leq \sqrt{\Lambda}\int^{1-\epsilon}_0\|R(T,U)X_t\|_Tdt,\tag{B.6}\label{B.8}
\end{align*}
where $\|\cdot\|$ is the norm induced by the average Riemannian metric $\tilde{g}$.

Let $\{e_i\}$ be a $g_T$-orthonormal basis for a fixed tangent space. Set $Z:=R(T,U)X_t$. Lemma \ref{CURVA} together with (\ref{*1}) furnishes
\begin{align*}
\|Z\|_T&  \leq \sum_i \|g_T(Z,e_i)\cdot e_i\|_T=\sum_i |R_T(X_t,e_i ,T,U)|\\
&= \sum_i \left|R_T\left(\frac{X_t}{F(X_t)},\frac{e_i}{F(e_i)}, \frac{T}{F(T)},\frac{U}{F(U)}\right)\right|\cdot F(X_t)\cdot F(e_i)\cdot F(T)\cdot F(U)\\
&\leq C_1\cdot F(X)\cdot R^2,
\end{align*}
where $C_1=C_1(n,k,\Lambda)$ is a constant. (\ref{B.8}) then implies
\[
\limsup_{t\rightarrow 1^-}\|\nabla_sX_t\|_U\leq C_2\cdot F(X)\cdot R^2,\tag{B.7}\label{B.9}
\]
where $C_2=C_2(n,k,\Lambda)$ is a constant.

Using (\ref{B.6}) and (\ref{B.9}), we have
\begin{align*}
&F(X(1,1)-X(0,1))
\leq \int^1_0 F\left(\frac{d}{ds}X_1(s)\right)ds\\
=&\int^1_0\lim_{t\rightarrow 1^-}F\left(\nabla_sX_t\right)ds
\leq \sqrt{\Lambda}\int^1_0\limsup_{t\rightarrow 1^-}\|\nabla_sX_t\|ds\leq  C_3\cdot F(X)\cdot R^2,
\end{align*}
where $C_3=\sqrt{\Lambda}\cdot C_2$.
\end{proof}


\begin{thebibliography}{10}





\bibitem[AM]{AM} U. Abresch and W.T. Meyer, \textsl{Injectivity radius estimates and sphere theorems}, pp. 1-47
in: Comparison geometry, edited by K. Grove et al., MSRI publications, volume
30, Cambridge 1997.



\bibitem[BC]{BC1}
 D. Bao and S. S. Chern, \textsl{A note on the Gauss-Bonnet theorem for Finsler spaces},  Ann. Math., \textbf{143}(1996), 233-252.

\bibitem[BCS]{BCS} D. Bao, S. S. Chern and Z. Shen, \textsl{An introduction
to Riemannian-Finsler geometry}, GTM {\bf{200}}, Springer-Verlag,
2000.



\bibitem[BBI]{DYS}D. Burago, Y. Burago and S. Ivanov, \textsl{A course in metric
geometry}, American Mathematical Society, 2001.











\bibitem[Cha]{C2} I. Chavel, \textsl{Riemannian Geometry - A Modern Introduction}, Academic Press, New York, 1984.

\bibitem[Che]{Ch} J. Cheeger, \textsl{Finiteness theorems for Riemannian manifolds}, Amer. J. Math., \textbf{92}(1970), 61-
74.











\bibitem[E]{E} D. Egloff, \textsl{Uniform Finsler Hadamard manifolds}, Ann. Inst. Henri Poincar¨¦, \textbf{66}(1997), 323-
357.


\bibitem[HK]{HK} E. Heintze and H. Karcher, \textsl{A general comparison theorem with applications to volume
estimates for submanifolds}, Ann. Sci. \'Ecole Norm. Sup. \textbf{II} (1978), 451-470.

\bibitem[Ka]{Ka}H. Karcher, \textsl{Riemannian center of mass and mollifier smoothing}, Comm. Pure
Appl. Math., \textbf{30}(1977), 509-541.


\bibitem[Kl]{K} W. Klingerberg, \textsl{Contributions to riemannian geometry in the large}, Ann. of Math.,
\textbf{69}(1959), 654-666.

\bibitem[Pe]{Pe} S. Peters, \textsl{Cheeger's finiteness theorem for diffeomorphism classes of Riemannian
manifolds}, J. Reine Angew. Math., \textbf{349} (1984), 77-82.


\bibitem[PP]{P} P. Petersen, \textsl{Riemannian Geometry}, second edition. Graduate Texts in Mathematics,
171. Springer, New York, 2006

 \bibitem[R1]{R3} H. Rademacher,   \textsl{A Sphere Theorem for non-reversible Finsler Metrics}, Math. Ann. \textbf{328}(2004), 373-387.



\bibitem[R2]{R} H. Rademacher, \textsl{Nonreversible Finsler metrics of positive
ag curvature}, A sampler
of Riemann-Finsler geometry, Cambridge Univ. Press, Cambridge, 2004, 261-302.

\bibitem[R3]{R2} H. Rademacher, \textsl{The length of a shortest geodesic loop},   C. R. Math., (13) \textbf{346}(2008), 763-765.








\bibitem[S]{Sh1} Z. Shen, \textsl{Lectures on Finsler geometry}, World
Sci., Singapore, 2001.









\bibitem[YZ]{YZ} L. Yuan and W. Zhao, \textsl{Some formulas of Santal\'o type in Finsler geometry and its
applications}, Publ. Math. Debrecen, accepted.



\bibitem[ZS]{ZY} W. Zhao and Y. Shen, \textsl{A Universal Volume Comparison Theorem for Finsler Manifolds and Related Results}, Can. J. Math., \textbf{65}(2013), 1401-1435.

\bibitem[Z1]{Z} W. Zhao, \textsl{Homotopy finiteness theorems for Finsler manifolds}, Publ. Math. Debrecen, \textbf{83}(2013), 329-358.

 \bibitem[Z2]{Z2} W. Zhao, \textsl{A Lower Bound for the length of closed Geodesics on a Finsler Manifold}, Canadian Mathematical Bulletin, \textbf{75}(2014), 194-208.


\end{thebibliography}
\end{document}